\documentclass[12pt]{amsart}
\usepackage{amsmath,amssymb,latexsym, amsthm, amscd, mathrsfs, stmaryrd} %, txfonts}

\usepackage[all]{xy}
\usepackage{hyperref}
\usepackage{color}
\hypersetup{colorlinks,linkcolor=blue,urlcolor=cyan,citecolor=blue}

\newcommand{\nc}{\newcommand}
\nc{\browntext}[1]{\textcolor{brown}{#1}}
\nc{\greentext}[1]{\textcolor{green}{#1}}
\nc{\redtext}[1]{\textcolor{red}{#1}}
\nc{\bluetext}[1]{\textcolor{blue}{#1}}
\nc{\brown}[1]{\browntext{ #1}}
\nc{\green}[1]{\greentext{ #1}}
\nc{\red}[1]{\redtext{ #1}}
\nc{\blue}[1]{\bluetext{ #1}}
\nc{\zb}[1]{\redtext{From zb: #1}}
\def \bi{{\imath}}
\newcommand{\ff}{B}
\newcommand{\tK}{\widetilde{K}}

\newtheorem{thm}{Theorem}  [section]
\newtheorem{cor}[thm]{Corollary}
\newtheorem{lem}[thm]{Lemma}
\newtheorem{prop}[thm]{Proposition}
\newtheorem{definition}[thm]{Definition}

\theoremstyle{remark}
\newtheorem{rem}{Remark}[section]

\numberwithin{equation}{section}

\newcommand{\mbf}{\mathbf}

\newcommand{\mrm}{\mathrm}
\newcommand{\A}{\mathcal A}

\newcommand{\ev}{\bar{0}}
\newcommand{\odd}{\bar{1}}

\newcommand{\diag}{\mrm{diag}}

\newcommand{\la}{\lambda}

\newcommand{\F}{\mathbb F}
\newcommand{\N}{\mathbb N}
\newcommand{\K}{\mathbb K}
\newcommand{\id}{\text{id}}

\newcommand{\bbZ}{\mathbb Z}

\newcommand{\one}{\mathbf 1}
\newcommand{\oldone}{\mathbf 1}
\newcommand{\onestar}{\one^\star}
\newcommand{\ov}{\overline}
\newcommand{\qbinom}[2]{\begin{bmatrix} #1\\#2 \end{bmatrix} }

\newcommand{\U}{\mbf U}
\newcommand{\Udot}{\dot{\mbf U}}

\newcommand{\Ui}{{\mbf U}^\imath}

\newcommand{\vs}{\varsigma}

\newcommand{\Z}{\mathbb Z}
\newcommand{\I}{\mathbb I}

\nc{\fprime}{\bold{'f}}
\def \m{{m}}

\def \bU{{\mathbf U}}

\begin{document}

\title[A Serre presentation for the $\imath${}quantum groups]{A Serre presentation for the $\imath${}quantum groups}

\author[Xinhong Chen]{Xinhong Chen}
\address{Department of Mathematics, Southwest Jiaotong University, Chengdu 610031, P.R.China}
\email{chenxinhong@swjtu.edu.cn}

\author[Ming Lu]{Ming Lu}
\address{Department of Mathematics, Sichuan University, Chengdu 610064, P.R.China}
\email{luming@scu.edu.cn}

\author[Weiqiang Wang]{Weiqiang Wang}
\address{Department of Mathematics, University of Virginia, Charlottesville, VA 22904}
\email{ww9c@virginia.edu}



%%%%%%%
\begin{abstract}
Let $(\bf U, \bf U^\imath)$ be a quasi-split quantum symmetric pair of arbitrary Kac-Moody type, where ``quasi-split" means the corresponding Satake diagram contains no black node. We give a presentation of the $\imath${}quantum group $\bf U^\imath$ with explicit $\imath$Serre relations. The verification of new $\imath$Serre relations is reduced to some new $q$-binomial identities. Consequently, $\bf U^\imath$ is shown to admit a bar involution under suitable conditions on the parameters. \end{abstract}
\maketitle

\setcounter{tocdepth}{1}
 \tableofcontents

%%%%%%%
%%%%%%%
\section{Introduction}

\subsection{}

Let $\U$ be a Drinfeld-Jimbo quantum group with Chevalley generators $E_i, F_i, K_i^{\pm 1}$, for $i\in I$. It is a $q$-deformation of the universal enveloping algebra of a symmetrizable Kac-Moody algebra with a Serre presentation. In terms of divided powers $F_i^{(n)} =F_i^n/[n]_{q_i}^!$ (cf. \cite{Lu93}; see \eqref{eq:qbinom} and its subsequent paragraph for notation $[n]_{q_i}^!$), the $q$-Serre relations among $F_i$'s can be written in a compact form as follows: for $i \neq j \in I$,
\begin{align}
  \label{eq:qSerre}
\sum_{n=0}^{1-a_{ij}} (-1)^n   F_i^{(n)}F_j F_i^{(1-a_{ij}-n)}=0 .
\end{align}
The quantum group $\U$ is a Hopf algebra with a comultiplication which is denoted by $\Delta$.

Quantum symmetric pairs (QSP for short), $(\U, \Ui)$, are deformations of symmetric pairs which are defined using Satake diagrams as the input, and $\Ui$ satisfies the coideal subalgebra property $\Delta: \Ui \rightarrow \Ui \otimes \U$. The theory of QSP was systematically studied by Letzter for $\U$ of finite type  (cf. \cite{Le99, Le02} for historical remarks and references therein). The QSP of Kac-Moody type was subsequently developed by Kolb \cite{Ko14}, unifying various special cases beyond finite type considered in the literature, some of which we mention below. We remark that the algebra $\Ui=\Ui_{\vs, \kappa}$ actually depends on a number of parameters $\vs =(\vs_i)_{i\in I}, \kappa =(\kappa_i)_{i\in I}$; see \eqref{parameters}. For example, some main generators of $\Ui$ are of the form, cf. \eqref{eq:def:ff}:
\[
B_i =F_{i} + \vs_i E_{\tau i} \tK^{-1}_i + \kappa_i \tK^{-1}_{i}, \quad \text{ for }i\in I,
\]
%\red{defined in \eqref{eq:def:ff}; see \eqref{Q4} for the definition of $\tK_i$, and see \eqref{eq:tau} for the definition of $\tau$. }
where the definitions of $\tK_i$ and the involution $\tau$ can be found in \eqref{Q4} and \eqref{eq:tau} respectively.
It has become increasingly clear in recent years (cf. \cite{BW13, BK15, BK19, BW18b} and the references therein) that the algebras $\Ui$ on their own are of fundamental importance, and we shall refer to them as the $\imath${}quantum groups.  %(instead of quantum symmetric pair coideal subalgebras).

Borrowing terminologies from real Lie groups, we shall call a quantum symmetric pair and an $\imath${}quantum group {\em quasi-split} (and respectively, {\em split}) if the underlying Satake diagram contains no black node (respectively, with the trivial involution in the Satake diagram). In other words, these are the $\imath${}quantum groups associated to the Chevalley involution $\omega$, coupled with a diagram involution $\tau$ (which is allowed to be the identity). Examples of the split $\imath${}quantum groups were considered in the literature (cf., e.g., \cite{Ter93, BaK05}) and they are also known as generalized $q$-Onsager algebras, cf. \cite{BaB10}. We refer to \cite[Introduction, (1)]{Ko14} for more detailed historical remarks. A quasi-split $\imath${}quantum group depends only on the generalized Cartan matrix and a diagram involution $\tau$.

Obtaining a nice presentation of the $\imath${}quantum group $\Ui$ is a fundamental problem, and it has useful applications. For example, to construct the bar involution on a general $\imath${}quantum group $\Ui$ as predicted in \cite{BW13}, one would need to have a {\em precise} presentation to see clearly what constraints on the parameters should be satisfied \cite{BK15}. The bar involution on $\Ui$ is a basic ingredient for the $\imath${}canonical basis \cite{BW18b}.

A presentation for a general $\Ui$ of finite type was given by Letzter \cite{Le02, Le03}.
Some less precise presentation for a general $\Ui$ of Kac-Moody type (where some Serre type relations were not explicit) was known earlier \cite{Ko14}; under the assumption that the Cartan integers $|a_{ij}| \le 3$, all the Serre type relations were found explicitly in terms of monomials in $B_i, B_j$ \cite{Le03, Ko14, BK15, BK19}, even though some of the formulas become complicated quickly as $|a_{ij}|$ increases; see the formulas \eqref{eq:123} in the quasi-split setting. A new and more conceptual approach is called for in order to reorganize and go beyond the known cases.

\subsection{}

The main result of this paper is a Serre presentation with precise and uniform relations for the quasi-split $\imath${}quantum groups of {\em arbitrary} Kac-Moody type with general parameters; see Theorems~\ref{thm:Serre} and \ref{thm:split}. One may view this work as an application of $\imath${}canonical bases to the foundational questions for the $\imath${}quantum groups.

The key to our Serre presentation is the so-called new {\em $\imath${}Serre relations} between $B_i$ and $B_j$ for $\tau i =i\neq j$. They are expressed in terms of the $\imath${}divided powers $B_{i,\overline{p}}^{(m)}$, for any fixed $\overline{p}\in \Z_2 =\{\bar 0, \bar 1\}$, as follows:
\begin{equation}
  \label{eq:iSe}
\sum_{n=0}^{1-a_{ij}} (-1)^n  B_{i,\overline{a_{ij}}+\overline{p}}^{(n)}B_j B_{i,\overline{p}}^{(1-a_{ij}-n)}  =0.
\end{equation}
Note that this relation formally takes the same form as the standard $q$-Serre relation \eqref{eq:qSerre}. Actually the relation \eqref{eq:iSe} holds for general (beyond quasi-split) $\imath$quantum groups; cf. Remark~\ref{rem:split+}. Let us explain the $\imath${}divided powers.

For distinguished parameters $\vs^\diamond$, i.e., $\vs_i^\diamond =q_i^{-1}$ for $i\in I$ such that $\tau i=i$ (cf. \eqref{eq:par0}), and $\kappa_j=0$ for all $j\in I$, the $\imath${}divided powers $B^{(m)}_{i,\ov{p}}$, for $\ov{p}\in \Z_2$ and $m \ge 1$, are explicit polynomials in $B_i$ introduced in \cite{BW13, BeW18} which depend on a parity $\ov p$ (arising from the parities of the highest weights of  highest weight $\U$-modules when evaluated at the coroot $h_i$). The $\imath${}divided powers are $\imath${}canonical basis elements for (the modified form of) $\Ui$ in the sense of \cite{BW18b}, but we will not need this fact here. For $\Ui$ with general parameters, we define $\imath${}divided powers as suitable polynomials in $B_i$ which are obtained via some rescaling isomorphism from those associated to the distinguished parameters; see \eqref{eq:iDPodd}--\eqref{eq:iDPev}. We caution that the $\imath$divided powers for more general parameters may not necessarily be $\imath$canonical basis elements.

It is instructive for the reader to verify that our $\imath${}Serre relations \eqref{eq:iSe} provide a uniform reformulation of the case-by-case complicated relations in \eqref{eq:123} (due to \cite{Ko14, BK19}) when $|a_{ij}|\le 3$. For illustration we convert the $\imath${}Serre relation for $a_{ij}=-4$ into a formula \eqref{eq:a=4} in terms of monomials in $B_i, B_j$, and compare it to \cite[(2.1)]{BaB10}.

The Serre presentation of $\Ui$ is valid in the specialization at $q=1$, providing a presentation of the fixed point Lie subalgebra $\mathfrak g^{\tau \omega}$ of a symmetrizable Kac-Moody algebra $\mathfrak g$, where $\omega$ is the Chevalley involution. Independently, when $\tau=1$ Stokman \cite{St18} gives a presentation of $\mathfrak g^\omega$ and calls it a generalized Onsager algebra, where the (classical) $\imath${}Serre relations are determined recursively; in contrast our formula for the (calssical) $\imath${}Serre is closed. One verifies that our $\imath${}Serre relations for $|a_{ij}|\le 4$ specialize to his formulas; cf. Remark~\ref{rem:q=1}.

\subsection{}

Our strategy of establishing the $\imath${}Serre relations \eqref{eq:iSe} is as follows. As $\Ui$ with arbitrary parameters is isomorphic to $\Ui$ with distinguished parameters $\vs^\diamond$ in \eqref{eq:par0} and $\kappa={\bf 0}$ (cf. \cite{Le02, Ko14}), we are reduced to the case of $\Ui$ with distinguished parameters.  As $\Ui$ is embedded in $\U$, it suffices to verify the identity \eqref{eq:iSe} in the quantum group $\U$, or alternatively, in the modified form $\Udot$.

To that end, the explicit expansion formulas of $\imath${}divided powers in terms of PBW basis of $\Udot$ given in \cite{BeW18} play a crucial role. The identity \eqref{eq:iSe} is reduced to the assertion of various coefficients in the PBW basis expansion of the left-hand side of \eqref{eq:iSe} are zero. After  reorganizations of the computations, the vanishings of all these coefficients somewhat miraculously reduce to  a universal $q$-binomial identity
\[
T(w, u, \ell) =0
\]
in 3 integral variables $w, u, \ell$; see \eqref{eq:Twux} for the definition of $T(w, u, \ell)$.

It turns out to be difficult to prove this $q$-binomial identity directly (the reader is encouraged to try his hand on this). We find a way around by first generalizing the $q$-identity to some $q$-identity involving 3 additional variables, such as, for $\ell >0$,
\[
G(w,u,  \ell;p_0,p_1,p_2)=0.
\]
See \eqref{eq:Twxyu} for definition of $G$.
The additional variables allow us to write down simple recursive relations, and then the (generalized) $q$-identity follows.

\subsection{} %Applications}

Let us indicate several applications of the main results of this paper.

The $\imath${}divided powers and $\imath${}Serre relations can be used to describe the higher order $\imath${}Serre relations and construct the braid group actions on $\Ui$; for the quantum group $\U$ this was done in \cite{Lu93}. This will be treated in a sequel to this paper \cite{CLW2}.

The bar involution on a quasi-split $\imath${}quantum group $\Ui$ is an indispensable ingredient for the $\imath${}canonical basis of the modified form of $\Ui$ developed in \cite{BW18c}.

The elegant $\imath${}Serre relations \eqref{eq:iSe} were motivated by and have in turn suggested a potential categorical interpretation \`a la Khovanov-Lauda; they are expected to play a fundamental role in the categorification of the quasi-split (modified) $\imath${}quantum groups.

Having the Serre presentations of quasi-split $\imath${}quantum groups available, one may hope that explicit $\imath${}Serre relations for general $\imath${}quantum groups $\Ui$ will be eventually written down in terms of $\imath${}canonical bases. To that end, more formulas for $\imath${}canonical bases need to be computed first, which are very challenging as the $\imath${}divided powers above already indicate.

\subsection{}

The paper is organized as follows. In Section~ \ref{section:preliminaries} we review and fix notations
for quantum groups and quantum symmetric pairs. %We formulate certain idempotents in the modified quantum group $\Ui$ for use in later sections.
In Section~\ref{section:main}, we formulate our main results and the main steps of their proofs. The Serre presentation of $\Ui$ can be found in Theorem~\ref{thm:Serre}. The $q$-binomial identity which is used to derive the $\imath${}Serre relations is stated as Theorem~\ref{thm:T=0}. The bar involution on $\Ui$ with suitable conditions on parameters specified is formulated as Proposition~\ref{prop:bar}.

Section \ref{sec:Serre=T}, Section ~\ref{sec:T=0}, and Appendix~\ref{sec:appendix} form the technical parts of the paper. In Section~\ref{sec:Serre=T}, we reduce the proof of the $\imath${}Serre relations to the $q$-binomial identity $T(w, u, \ell) =0$; some additional reduction steps are collected in Appendix~\ref{sec:appendix}. In Section ~\ref{sec:T=0}, we formulate and prove a generalization of the $q$-identity $T(w, u, \ell) =0$; in particular, this identity follows.

\subsection*{Acknowledgments.}
XC is supported partially by the National Natural Science Foundation of China grant No. 11601441 and the Fundamental Research Funds for the Central Universities grant No. 2682016CX109. We thank University of Virginia and East China Normal University for hospitality during our visits when this project is carried out. WW is partially supported by NSF grant DMS-1702254. We thank the 2 anonymous referees for helpful comments and corrections.

%%%%%%%%
%%%%%%%%
\section{The preliminaries}
    \label{section:preliminaries}

\subsection{Quantum groups}
\label{subsec:QG}

We recall the definitions of Cartan datum and root datum from \cite[1.1.1, 2.2.1]{Lu93}.
A \emph{Cartan datum} is a pair $(I,\cdot)$ consisting of a finite set $I$ and a symmetric bilinear form $\nu,\nu'\mapsto \nu\cdot \nu'$ on the free abelian group $\Z[I]$ such that
\begin{itemize}
\item[(a)] $i\cdot i\in\{2,4,6,\dots\}$ for any $i\in I$;
\item[(b)] $2\frac{i\cdot j}{i\cdot i}\in\{0,-1,-2,\dots\}$ for any $i\neq j$ in $I$.
\end{itemize}

A \emph{root datum} of type $(I,\cdot)$ consists of
\begin{itemize}
\item[(a)] two finitely generated free abelian groups $Y,X$ and a perfect bilinear pairing $\langle\cdot,\cdot\rangle:Y\times X\rightarrow\Z$;
\item[(b)] an embedding $I\subset X$ ($i\mapsto \alpha_i$) and an embedding $I\subset Y$ ($i\mapsto h_i$) such that $\langle h_i,\alpha_j\rangle =2\frac{i\cdot j}{i\cdot i}$ for all $i,j\in I$.
\end{itemize}
%$X$ is called the \emph{weight lattice} and $Y$ is called the \emph{root lattice}.
The matrix $A:= (a_{ij}):=(\langle h_i, \alpha_j\rangle)$ is a \emph{symmetrizable generalised Cartan matrix}.
%In fact, $A$ is a square matrix with integer entries $a_{ij}$ which satisfies the following three conditions:
%\begin{itemize}
%\item[(C1)] $a_{ii}=2$ for all $i\in I$,
%\item[(C2)] $a_{ij}leq 0$ if $i\neq j$,
%\item[(C3)] $a_{ij}=0$ if and only if $a_{ji}=0$.
%\end{itemize}
Let
\[
D=\diag(\epsilon_i\mid i\in I), \qquad \text{ where } \epsilon_i=\frac{i\cdot i}{2}  \; (\forall i\in I).
\]
 Then $DA$ is symmetric.
We shall assume that the root datum defined above is \emph{$X$-regular} and \emph{$Y$-regular}, that is,
$\{\alpha_i\mid i\in I\}$ is linearly independent in $X$ and $\{h_i\mid i\in I\}$ is linearly independent in $Y$.
%
%We define a partial order $\leq$ on $X$ as follows: for $\la, \la' \in X$, $\lambda \le \lambda' \text{ if and only if } \lambda' -\lambda \in \N[I].$

%
Let $q$ be an indeterminate, and denote
\[
q_i:=q^{\epsilon_i}, \qquad \forall i\in I.
\]
For $n,d, m\in \Z$ with $m\ge 0$, we denote the $q$-integers and $q$-binomial coefficients as
\begin{align}
  \label{eq:qbinom}
  \begin{split}
[n] =\frac{q^n-q^{-n}}{q-q^{-1}},
\qquad
[m]! & = [1][2]\cdots [m],
\\
\qbinom{n}{d}  &=
\begin{cases}
\frac{[n][n-1]\ldots [n-d+1]}{[d]!}, & \text{ if }d \ge 0,
\\
0, & \text{ if }d<0.
\end{cases}
\end{split}
\end{align}
We denote by $[n]_{q_i}, [m]_{q_i}^!,$ and $\qbinom{n}{m}_{q_i}$ the variants of $[n], [m]!,$ and $\qbinom{n}{m}$ with $q$ replaced by $q_i$.
For  any $i\neq j\in I$, define the following polynomial in two (noncommutative) variables
\begin{equation}
\label{eq:Sij}
S_{ij}(x,y)=\sum_{n=0}^{1-a_{ij}} (-1)^n \qbinom{1-a_{ij}}{n}_{q_i} x^{n}yx^{1-a_{ij}-n}.
\end{equation}

Let $\K$ be a field of characteristic $0$. %Let $\K(q)$  be the field of rational functions in an indeterminate $q$.
Assume that a root datum $(Y,X,\langle\cdot,\cdot\rangle,\dots)$ of type $(I,\cdot)$ is given.
The quantum group $\bU$ is the associative $\K(q)$-algebra with generators $E_i,F_i,K_h$ for all $i,j\in I$ and $h\in Y$ subject to the following relations:
\begin{align}
&K_0=1,\; K_hK_{h'}=K_{h+h'}, \quad \forall h,h'\in Y.\label{Q1}
\\
&K_hE_i=q^{\langle h,\alpha_i\rangle} E_iK_h, \quad \forall i\in I, h\in Y.\\
&K_hF_i=q^{-\langle h,\alpha_i\rangle}F_iK_h, \quad \forall i\in I, h\in Y.\\
&[E_i,F_j]=\delta_{ij}\frac{\widetilde{K}_i-\widetilde{K}_i^{-1}}{q_i-q_i^{-1}}, \text{ where }\widetilde{K}_i=K_{h_i}^{\epsilon_i}, \forall i\in I.\label{Q4}\\
& (q\text{-Serre relations})\,\,
S_{ij}(E_i,E_j)=0=S_{ij}(F_i,F_j), \quad \forall i \neq j\in I.\label{q serre}
\end{align}

Let
\[
F_i^{(n)} =F_i^n/[n]_{q_i}^!, \quad E_i^{(n)} =E_i^n/[n]_{q_i}^!, \quad \text{ for } n\ge 1 \text{ and  } i\geq 1.
\]
Then the $q$-Serre relations \eqref{q serre} above
can be rewritten  as follows: for $i\neq j \in I$,
\begin{align*}
\sum_{n=0}^{1-a_{ij}} (-1)^n  E_i^{(n)} E_j E_i^{({1-a_{ij}-n})}=0,
\end{align*}
\begin{align}\label{eq:serre4}
\sum_{n=0}^{1-a_{ij}} (-1)^n  F_i^{(n)}F_j F_i^{(1-a_{ij}-n)}=0 .
\end{align}
Denote by $\omega$ the \emph{Chevalley involution}, which is the $\K(q)$-algebra automorphism of $\bU$ sending
\[
\omega(E_i)=-F_i, \quad \omega(F_i)=-E_i, \quad \omega(K_h)=K_{-h}.
\]
The following lemma is a higher rank generalization of the involution $\varpi$ on $\U_q(\mathfrak{sl}_2)$ defined in \cite[Remark 2.3]{BeW18}.
\begin{lem}   \label{lemma:involution of U}
There exists an involution $\varpi$ on the $\K$-algebra $\U$ which sends
\begin{align}
\varpi: E_i\mapsto q_i^{-1}F_i\tK_i,\quad F_i\mapsto q_i^{-1} E_i\tK_i^{-1},\quad K_\mu\mapsto K_\mu,\quad q\mapsto q^{-1}.
\end{align}
for any $i\in I$, $\mu\in Y$.
\end{lem}

\begin{proof}
Knowing that the rank one relations are preserved, we see quickly that $\varpi$ preserves the defining relations \eqref{Q1}--\eqref{Q4} for $\U$. It remains to show that $\varpi$ preserves the $q$-Serre relations \eqref{q serre}: $S_{ij} (q_i^{-1}F_i\tK_i, q_j^{-1}F_j\tK_j)=0$ and $S_{ij} (q_i^{-1}E_i\tK_i^{-1}, q_j^{-1}E_j\tK_j^{-1})=0$, for $i\neq j\in I$. This is known, and for the sake of completeness let us include a short argument:
\begin{align*}
S_{ij} (q_i^{-1}F_i\tK_i, q_j^{-1}F_j\tK_j)
&= q_i^{a_{ij}-1}q_j^{-1} \sum_{n=0}^{1-a_{ij}} (-1)^n \qbinom{1-a_{ij}}{n}_{q_i} (F_i\tK_i)^{n} (F_j\tK_j) (F_i\tK_i)^{1-a_{ij}-n}
\\
%&= q_i^{a_{ij}-1}q_j^{-1} \sum_{n=0}^{1-a_{ij}} (-1)^n \qbinom{1-a_{ij}}{n}_{q_i} F_i^{1-a_{ij}-n} F_j F_i^n \tK_i^{1-a_{ij}}\tK_j \\
&= q_i^{a_{ij}-1}q_j^{-1} S_{ij}(F_i,F_j) \tK_i^{1-a_{ij}}\tK_j =0.
\end{align*}
The other relation is entirely similar.
\end{proof}

Let $\bU^+,\bU^-$ and $\bU^0$ be the subalgebra of $\bU$ generated by
$\{E_i\mid i\in I\}$, $\{F_i\mid i\in I\}$ and $\{K_h\mid h\in Y\}$ respectively.
%Then the multiplication map give an isomorphism of vector spaces $$\bU^+\otimes \bU^0\otimes \bU^-\cong \bU.$$

%%
\subsection{The algebra $\Udot$}
   \label{subsection Udot}

Recall \cite[23.1]{Lu93} that  the modified form of $\bU$, denoted by $\Udot$, is a $\K(q)$-algebra (without 1)  generated by $\oldone_\la, E_i \oldone_\la$, $F_i \oldone_\la$, for $i\in I, \la \in X$, where $\oldone_\la$ are orthogonal idempotents. Let
$
\A=\Z[q,q^{-1}].
$
There is an $\A$-subalgebra  $_\A \Udot$ generated by $E_i^{(n)}\oldone_\la, F_i^{(n)}\oldone_\la$ for $i\in I$ and $n\geq 0$ and $\la \in X$.
Note that $\Udot$ is naturally a $\bU$-bimodule \cite[23.1.3]{Lu93}, and in particular we have
\begin{align*}
%E_i^{(a)}\oldone_\la =\oldone_{\la+a\alpha_i} E_i^{(a)}, & \qquad F_i^{(a)}\oldone_\la=\oldone_{\la-a\alpha_i} F_i^{(a)};
% \\
K_h\oldone_\la=\oldone_\la K_h &=q^{\langle h,\la\rangle}\oldone_\la, \; \forall h\in Y.
%\\
%E_iF_j\oldone_\la-F_jE_i\oldone_\la &=\delta_{ij}[\langle h_i,\la\rangle]_{q_i}\oldone_\la;
%\\
%E_i^{(a)} F_j^{(b)}\oldone_\la &=F_j^{(b)} E_i^{(a)}\oldone_\la, \; \forall i\neq j.
\end{align*}

We have the mod $2$ homomorphism $\Z \rightarrow \Z_2, k \mapsto \ov k$, where $\Z_2=\{\bar{0},\bar{1}\}$. Let us fix an $i\in I$.  Define
\begin{equation}
\Udot_{i,\ev}:=\bigoplus_{\la:\, \langle h_i,\la\rangle \in 2\Z}\,\Udot \oldone_{\la},
\qquad
\Udot_{i,\odd}:=\bigoplus_{\la:\, \langle h_i,\la\rangle \in1+2\Z}\, \Udot\oldone_{\la}.
\end{equation}
Then $\Udot= \,\Udot_{i,\ev}\oplus{\Udot_{i,\odd}}$.
Similarly, letting $_\A\Udot_{i,\ev} =\Udot_{i,\ev} \cap _\A\Udot$ and $_\A\Udot_{i,\odd} =\Udot_{i,\odd} \cap _\A\Udot$, we have $_\A\Udot = {}_\A\Udot_{i,\ev} \oplus {}_\A\Udot_{i,\odd}$.

For our later use, with $i\in \I$ fixed once for all, we need to keep track of the precise value $\langle h_i,\la\rangle$ in an idempotent $\oldone_\la$ but do not need to  know which specific weights $\la$ are used. Thus it is convenient to introduce the following generic notation
\begin{align}
 \label{eq:1star}
\onestar_m =\onestar_{i,m}, \qquad \text{ for }m\in \Z,
\end{align}
to denote an idempotent $\oldone_\la$ for some $\la\in X$ such that $\m=\langle h_i,\la\rangle$.
In this notation, the identities in \cite[23.1.3]{Lu93} can be written as follows: for any $\m\in\Z$, $a,b\in\Z_{\geq0}$, and $i\neq j\in I$,
\begin{align}
E_i^{(a)}\onestar_{i,\m} &=\onestar_{i,\m+2a} E_i^{(a)},\quad F_i^{(a)}\onestar_{i,\m}=\onestar_{i,\m-2a} F_i^{(a)};
\label{eqn: idempotent Ei Fi}\\
E_j\onestar_{i,\m} &=\onestar_{i,\m+a_{ij}}E_j,  \qquad F_j\onestar_{i,\m}=\onestar_{i,\m-a_{ij}} F_j;
  \label{eqn: idempotent Ej Fj}\\
F_{i}^{(a)} E_i^{(b)}\onestar_{i,\m} &= \sum_{j=0}^{\min\{a,b\}} \qbinom{a-b-\m}{j}_{q_i} E_i^{(b-j)} F_i^{(a-j)}\onestar_{i,\m};
  \label{eqn:commutate-idempotent3}\\
E_i^{(a)} F_i^{(b)}\onestar_{i,\m} &=\sum_{j=0}^{\min\{a,b\}} \qbinom{a-b+\m}{j}_{q_i} F_i^{(b-j)} E_i^{(a-j)}\onestar_{i,\m}
  \label{eqn:commutate-idempotent4}.
\end{align}
From now on, we shall always drop the index $i$ to write the idempotents as $\onestar_m$.

\begin{rem}
  \label{rem:u=0}
If $u\in\U$ satisfies $u\onestar_{2k-1}=0$ for all possible idempotents $\onestar_{2k-1}$ with $k\in\Z$ (or respectively,  $u\onestar_{2k}=0$ for all possible $\onestar_{2k-1}$ with $k\in\Z$), then $u=0$.
\end{rem}

\subsection{The $\imath${}quantum group $\Ui$}
  \label{subsec:irootdatum}

Let $(Y,X, \langle\cdot,\cdot\rangle, \cdots)$ be a root datum of type $(I, \cdot)$.  We call a permutation $\tau$ of the set $I$ an {\em involution} of the Cartan datum $(I, \cdot)$ if $\tau^2 =\id$ and $\tau i \cdot \tau j = i \cdot j$ for $i$, $j \in I$.  Note we allow $\tau =\id$.
We shall always assume that $\tau$ extends to  an involution on $X$ and an involution on $Y$ (also denoted by $\tau$), respectively, such that the perfect bilinear pairing is invariant under the involution $\tau$.
The permutation $\tau$ of $I$ induces an $\K(q)$-algebra automorphism of $\U$, %denoted also by $\tau$,
defined by
\begin{align}
 \label{eq:tau}
\tau:E_i \mapsto E_{\tau i}, \quad F_i \mapsto F_{\tau i}, \quad K_h \mapsto K_{\tau h}, \qquad \forall i\in I, h\in Y.
\end{align}

Define %Following \cite[\S3.1]{BW18b}, we introduce
\begin{align}
  \label{XY}
 \begin{split}
%X_{{\imath}} = X /  \breve{X}, & \quad \text{ where } \; \breve{X}  = \{ \la + \tau(\la) \mid \la \in X\}, \\
Y^{\imath} &= \{h \in Y  \mid \tau(h)=-h \}.
\end{split}
\end{align}
%We shall call $X_{\imath}$ the {\em $\imath${}weight lattice} (even though $X_{\imath}$ is {\em not} always a lattice), and call $Y^{\imath}$ the {\em $\imath${}root lattice}, respectively.
%For any $\la \in X$ denote its image in $X_{\imath}$ by $\overline{\la}$. There is a well-defined bilinear pairing (denoted by $\langle \cdot, \cdot \rangle$ again, by abuse of notations)
%\[
%\langle \cdot, \cdot \rangle:
%Y^{\imath}  \times X_{\imath} \longrightarrow \Z
%\]
%defined by $\langle h, \overline{\la} \rangle  : = \langle h, \la \rangle$, where $\la \in X$ is any preimage of $\overline{\la}$ and $h \in Y^{\imath}$.

In this paper we will only consider a subclass of quantum symmetric pairs defined in \cite{Le99, Ko14}
 (which correspond to Satake diagrams without black nodes).

\begin{definition}  [\cite{Le99, BaB10, Ko14}]
  \label{def:Ui}
The {\em quasi-split $\imath${}quantum group}, denoted by $\Ui_{\vs,\kappa}$ or $\Ui$, is the $\K(q)$-subalgebra of $\U$ generated by
\begin{align}
B_i :=F_{i}  &+ \vs_i E_{\tau i} \tK^{-1}_i + \kappa_i \tK^{-1}_{i} \, \,(i \in I),
 \qquad K_{\mu}\, \,(\mu \in Y^{\imath}).
  \label{eq:def:ff}
\end{align}
Here the parameters
\begin{equation}
  \label{parameters}
  \vs=(\vs_i)_{i\in I}\in (\K(q)^\times)^I,\qquad
  \kappa=(\kappa_i)_{i\in I}\in\K(q)^I
\end{equation}
are assumed to satisfy Conditions \eqref{kappa}--\eqref{vs=} below:
\begin{align}
 \label{kappa}
\kappa_i &=0 \; \text{ unless } \tau i =i \text{ and } \langle h_k,\alpha_i \rangle \in 2\Z \; \forall k = \tau(k);
\\
%\vs_{{\tau i}} &= q_i^{-a_{i,\tau i}} \overline{\tau_i} \text{ if }    a_{i,\tau i} \neq 0;   \label{vs=bar}
%\\
%\overline{\kappa_i} &= \kappa_i;   \label{kappa2}\\
%
\vs_{i} & =\vs_{{\tau i}} \text{ if }    a_{i,\tau i} =0.
\label{vs=}
\end{align}
\end{definition}
The conditions on the parameters ensure that $\Ui$ has the expected size.

$\triangleright$ The pair $(\U, \Ui)$ forms a {\em quantum symmetric pair} (QSP) \cite{Le99, Ko14}, as its $q\mapsto 1$ limit is the classical symmetric pair and $\Ui$ is a (right) coideal subalgebra of $\U$, i.e., $\Delta: \Ui \longrightarrow \Ui \otimes \U$.

$\triangleright$ The $\imath${}quantum group $\Ui$ is also called a quantum symmetric pair coideal subalgebra in some papers.

$\triangleright$ We refer to this subclass of QSP $(\U, \Ui)$ or $\Ui$ in Definition~\ref{def:Ui} as {\em quasi-split},  which correspond to Satake diagrams without black nodes.

$\triangleright$ We call QSP $(\U, \Ui)$ or $\Ui$ above {\em split} if in addition $\tau=\id$, borrowing terminologies from the literature of real groups. They are also known as generalized $q$-Onsager algebras, cf. \cite{BaB10}.
In case when $\U$ is the quantum affine $\mathfrak{sl}_2$, $\Ui$ was known as $q$-Onsager algebras; cf. \cite{Ter93, BaK05}. Note that s split $\Ui$ is generated only by $B_i \, \,(i \in I)$.

%%%%%%%%%
%%%%%%%%%
\section{A Serre presentation of $\Ui$ and a $q$-binomial identity}
  \label{section:main}

\subsection{$\imath${}divided powers}

Assume that a root datum $(Y,X, \langle\cdot,\cdot\rangle, \cdots)$ is given.
Let $\Ui =\Ui_{\vs,\kappa}$ be an $\imath${}quantum group with parameters $(\vs, \kappa)$; cf. \S\ref{subsec:irootdatum}.
%$$  \vs=(\vs_i)_{i\in I}\in (\K(q)^\times)^I,\kappa=(\kappa_i)_{i\in I}\in\K(q)^I.$$

%Recall that $\Z_2=\{\ev,\odd\}$.

For  $i\in I$ with $\tau i\neq i$, imitating Lusztig's divided powers,  we define the {\em $\imath${}divided power} of $B_i$ to be (cf. \cite[(2.2)]{BW13})
\begin{align}
  \label{eq:iDP1}
  B_i^{(m)}:=B_i^{m}/[m]_{q_i}^!, \quad \forall m\ge 0, \qquad (\text{if } i \neq \tau i).
\end{align}
% It is observed in \cite[(2.2)]{BW13} that $B_i^{(m)}$ lies in ${}_\A \U$ thanks to the $q$-binomial formula. \red{Here $_\A \U$ is the $\A$-subalgebra of $\U$ generated by the elements $E_i^{(m)}$, $F_i^{(m)}$ for various $i\in I$ and $m\in\Z$ and by the elements $K_\mu$ for $\mu\in Y$.}
%
For $i\in I$ with $\tau i= i$, generalizing \cite{BW13}, we define the {\em $\imath${}divided powers} of $B_i$ to be
\begin{eqnarray}
&&\ff_{i,\odd}^{(m)}=\frac{1}{[m]_{q_i}^!}\left\{ \begin{array}{ccccc} B_i\prod_{j=1}^k (B_i^2-q_i\vs_i[2j-1]_{q_i}^2 ) & \text{if }m=2k+1,\\
\prod_{j=1}^k (B_i^2-q_i\vs_i[2j-1]_{q_i}^2) &\text{if }m=2k; \end{array}\right.
  \label{eq:iDPodd}\\
&&\ff_{i,\ev}^{(m)}= \frac{1}{[m]_{q_i}^!}\left\{ \begin{array}{ccccc} B_i\prod_{j=1}^k (B_i^2-q_i\vs_i[2j]_{q_i}^2 ) & \text{if }m=2k+1,\\
\prod_{j=1}^{k} (B_i^2-q_i\vs_i[2j-2]_{q_i}^2) &\text{if }m=2k, \end{array}\right.
 \label{eq:iDPev}.
\end{eqnarray}
In case when the parameter $\vs_i=q_i^{-1}$, the formulas \eqref{eq:iDPodd}--\eqref{eq:iDPev} first appeared in \cite[Conjecture~ 4.13]{BW13} (where $\kappa_i=1$) and were then studied in depth in \cite{BeW18} (where $\kappa_i=0, 1$). We shall see the formulas \eqref{eq:iDPodd}--\eqref{eq:iDPev} for general parameter $\vs_i$ arise from some rescaling isomorphism.

\subsection{A Serre presentation of $\Ui$}\label{subsection:iserre presentation}

Denote
\[
(a;x)_0=1, \qquad (a;x)_n =(1-a)(1-ax)  \cdots (1-ax^{n-1}), \quad \forall n\ge 1.
\]
Now we state our first main result. Let us fix $\ov{p}_i\in \Z_2$ for each $i\in I$.

\begin{thm}\label{thm:Serre}
The $\K(q)$-algebra $\Ui_{\vs,\kappa}$ has a presentation with generators $B_i$ $(i\in I)$, $K_\mu$ $(\mu\in Y^\imath)$ and the relations \eqref{relation1}--\eqref{relation2} below: for $\mu,\mu'\in Y^{\imath}$ and $i\neq j \in I$,
\begin{align}
K_{\mu}K_{-\mu} &=1,
 \quad
K_\mu K_{\mu'}=K_{\mu+\mu'},   \label{relation1}
\\
K_\mu B_i-q_i^{-\langle \mu,\alpha_i\rangle} B_iK_\mu & =0,
\\
B_iB_{j}-B_jB_i &=0, \quad \text{ if }a_{ij} =0 \text{ and }\tau i\neq j,
\\
\sum_{n=0}^{1-a_{ij}} (-1)^nB_i^{(n)}B_jB_i^{(1-a_{ij}-n)} &=0, \quad \text{ if } j \neq \tau i\neq i,
\\
\sum_{n=0}^{1-a_{i,\tau i}} (-1)^{n+a_{i,\tau i}} B_i^{(n)}B_{\tau i}B_i^{(1-a_{i,\tau i}-n)}& =\frac{1}{q_i-q_i^{-1}}
\label{relation5}     \\
 \cdot \left(q_i^{a_{i,\tau i}} (q_i^{-2};q_i^{-2})_{-a_{i,\tau i}} \vs_{\tau i}B_i^{(-a_{i,\tau i})} \tK_i \tK_{\tau i}^{-1} \right.
 & \left. -(q_i^{2};q_i^{2})_{-a_{i,\tau i}}\vs_{i}B_i^{(-a_{i,\tau i})} \tK_{\tau i} \tK_i^{-1} \right),
\text{ if } \tau i \neq i,
 \notag \\
\sum_{n=0}^{1-a_{ij}} (-1)^n  B_{i,\overline{a_{ij}}+\overline{p_i}}^{(n)}B_j B_{i,\overline{p}_i}^{(1-a_{ij}-n)} &=0,\quad   \text{ if }\tau i=i.
\label{relation2}
\end{align}
(This presentation will be called a {\em Serre presentation} of $\Ui$.)
\end{thm}
A proof of Theorem~\ref{thm:Serre} will be presented at the end of this section, \S\ref{subsec:proofSerre}.

\begin{rem}
  \label{rem:kappa}
There is a presentation of $\Ui$ in \cite[Theorem 7.4]{Le03} for finite type and a less precise one in \cite[Theorem~7.1]{Ko14} for Kac-Moody type, where relations \eqref{relation5}-\eqref{relation2} were replaced by some implicit identities of the form
\[
\sum_{n=0}^{1-a_{ij}} (-1)^n \qbinom{1-a_{ij}}{n}_{i} B_i^nB_jB_i^{1-a_{ij}-n} =C_{ij}
\]
where $C_{ij}$ are some suitable unspecified lower terms, for $j\neq i \in I$ with $\tau i=i$; see \cite[Theorem 3.6]{BK15} for an update, which establishes the explicit relation \eqref{relation5}. It follows by \cite[Theorem~ 7.1]{Ko14} that a  presentation of $\Ui$ in terms of $B_i$ and $K_\mu$ is independent of the parameters $\kappa_i \; (i\in I)$.

The \emph{$\imath${}Serre relation} \eqref{relation2} is a main novelty of this paper.
 \end{rem}

\begin{rem}
  \label{rem:aij123}
Under the assumption
\begin{equation}
  \label{eq:aij}
a_{ij}\in\{0,-1,-2,-3\}, \qquad \forall i\neq j\in I,
\end{equation}
a Serre presentation of $\Ui$ of Kac-Moody type has been given in \cite[Theorems 7.4, 7.8]{Ko14} (for $a_{ij} \in \{0,-1,-2\}$) and \cite[Theorem 3.7]{BK15} (for $a_{ij}=-3$),  where the $\imath${}Serre relation \eqref{relation2} was replaced by the following relations (some sign typos in \cite[Theorem 3.7]{BK15} when $a_{ij}=-3$ are corrected here):
\begin{align}
&\sum_{n=0}^{1-a_{ij}} (-1)^n \qbinom{1-a_{ij}}{n}_{q_i} B_i^{n}B_j B_i^{1-a_{ij}-n}
\label{eq:123} \\
&\quad =\left\{ \begin{array}{llc}
q_i\vs_i B_j,&\text{ if }a_{ij}=-1;\\
{} -[2]_{q_i}^2 q_i\vs_i (B_iB_j-B_jB_i),& \text{ if }a_{ij}=-2;
\\
{} -[2]_{q_i}([2]_{q_i}[4]_{q_i}+q_i^2+q_i^{-2})q_i\vs_i  B_iB_jB_i
& \\
\qquad +([3]_{q_i}^2+1)q_i\vs_i  (B_i^2B_j+B_jB_i^2)
-[3]_{q_i}^2(q_i\vs_i)^2B_j, & \text{ if }a_{ij}=-3.
\end{array} \right.
 \notag
\end{align}

We leave it to the reader to convert these complicated formulas to \eqref{relation2}, in 2 different forms with $\ov{p}_i\in \{\ov 0, \ov 1\}$, hence verifying Corollary~\ref{cor:S01} below directly in these cases.
\end{rem}

For $a_{ij}=-4$, the $\imath${}Serre relation \eqref{relation2} can be converted to
\begin{align}
 \sum_{n=0}^{5} (-1)^n  \qbinom{5}{n}_{q_i} B_i^{n}B_j B_i^{5-n}
 & = -[2]_{q_i}^2 (1+[2]_{q_i^2}^2)q_i\vs_i (B_i^3B_j -B_jB_i^3)
 \label{eq:a=4}
  \\
& \quad +[2]_{q_i}^2[5]_{q_i}[3]_{q_i} q_i\vs_i (B_i^2B_jB_i -B_iB_jB_i^2)  \notag
\\
&\quad +[2]_{q_i}^2[4]_{q_i}^2 (q_i\vs_i)^2(B_iB_j-B_jB_i).  \notag
\end{align}
This formula is compatible with \cite[(2.1)]{BaB10}, if the scalars $\rho_{ij}^0, \rho_{ij}^1$ therein are chosen to be $\rho_{ij}^0=-[2]_{q_i}^2 (1+[2]_{q_i^2}^2)q_i\vs_i$, and $\rho_{ij}^1=[2]_{q_i}^2[4]_{q_i}^2(q_i\vs_i)^2$.

\begin{rem}  \label{rem:q=1}
Let $\mathfrak g$ denote the Kac-Moody algebra associated to the generalized Cartan matrix $A=(a_{ij})$. Theorem~\ref{thm:Serre} specializes at $q=1$ to a presentation of the fixed point Lie subalgebra $\mathfrak g^{\tau \omega}$ of $\mathfrak g$, where $\omega$ is the Chevalley involution and $\tau$ is a Dynkin diagram involution. When $\tau=1$, a presentation for $\mathfrak g^{\omega}$ (called a generalized Onsager algebra) was independently obtained in a recent paper \cite{St18}, where the $\imath${}Serre relations are given by some recursive formulas. One verifies directly that the above formulas for $|a_{ij}|=3,4$ specialize at $q=1$ to \cite[(2.10)]{St18}. These Serre-type formulas (at $q=1$) must be a priori compatible by a uniqueness argument similar to the proof of Corollary~\ref{cor:S01} below.
\end{rem}

\begin{rem}
  \label{rem:split+}
The new $\imath${}Serre relations \eqref{relation2} remain valid for general $\imath${}quantum groups which are not quasi-split. More precisely they are valid under the assumption $\tau i=i$, $w_{\bullet} i=i$, but allowing possibly $w_{\bullet} \alpha_j\neq \alpha_j$; in this case, $B_j =F_j  + \vs_j \texttt{T}_{w_\bullet} (E_{\tau j}) \tK^{-1}_j + \kappa_i \tK^{-1}_{j}$; see \cite[Definition 3.5]{BW18b} for notations. See Remark~\ref{rem:equiv+} for an outline of a proof.
 \end{rem}

\begin{cor}
  \label{cor:S01}
For $j \neq i\in I$ with $\tau i=i$, we have
\begin{align}
   \label{eq:0=1}
\sum_{n=0}^{1-a_{ij}} (-1)^n  B_{i,\ov{a_{ij}}}^{(n)}B_j B_{i,\ov{0}}^{(1-a_{ij}-n)}
=\sum_{n=0}^{1-a_{ij}} (-1)^n  B_{i,\overline{a_{ij}+1}}^{(n)}B_j B_{i,\ov{1}}^{(1-a_{ij}-n)}
\end{align}
as polynomials in non-commutative variables $B_i$ and $B_j$.
\end{cor}

\begin{proof}
Let us denote the LHS and RHS of \eqref{eq:0=1} by $\mathcal S_{ij,0}$ and $\mathcal S_{ij,1}$, respectively. Note as polynomials in $B_i, B_j$, $[1-a_{ij}]_{q_i}^! \mathcal S_{ij,p}$ (for $p=0,1$) is of the form $S_{ij}(B_i, B_j) -C_{ij;p}$ (see \eqref{eq:Sij}, \eqref{eq:iDPodd}--\eqref{eq:iDPev} for notations), where $C_{ij;p}$ are some polynomials in $B_i, B_j$ over $\U^{\imath,0}$ of degree lower than $\deg S_{ij}(B_i, B_j)$; here $\U^{\imath,0}$ denotes the $\K(q)$-subalgebra of $\Ui$ generated by $K_{\mu}\, \,(\mu \in Y^{\imath})$.

Hence $[1-a_{ij}]_{q_i}^!(\mathcal S_{ij,0} - \mathcal S_{ij,1}) = C_{ij,0} -C_{ij,1}$ is a polynomial in $B_i, B_j$ over $\U^{\imath,0}$ of degree $< 2-a_{ij}$. By Theorem~\ref{thm:Serre}, $\mathcal S_{ij,p}=0$, for $p=0,1$, are relations in $\Ui$, and so is $C_{ij,0} -C_{ij,1}=0$. Recall from \cite[\S 7]{Ko14} that $\Ui$ has a filtration (roughly speaking by regarding $F_i$ as the highest term of $B_i$ \eqref{eq:def:ff}), whose associated graded is $\U^-$ ( over $\U^{\imath,0}$). If $C_{ij,0} -C_{ij,1}$ were a nonzero polynomial in $B_i, B_j$, then the relation $C_{ij,0} -C_{ij,1}=0$ in $\Ui$ would descend to a nontrivial relation in the associated graded between $F_i, F_j$ of degree below $\deg S_{ij}(F_i, F_j)$, %the degree of the standard $q$-Serre relation (Q5) in \S\ref{subsec:QG},
a contradiction. So as a polynomial in $B_i, B_j$ we have $C_{ij,0} =C_{ij,1}$, and hence, $\mathcal S_{ij,0} =\mathcal S_{ij,1}$.
\end{proof}

Recall a quasi-split $\imath${}quantum group $\Ui$ is split if $\tau=\id$. The Serre presentation for split $\Ui$ takes an especially simple form, which we record here.
\begin{thm}
  \label{thm:split}
Fix $\ov{p}_i \in \Z_2$, for each $i\in I$. Then the split $\imath${}quantum group $\Ui$ has a Serre presentation with generators $B_i$ $(i\in I)$
and relations
\begin{align*}
\sum_{n=0}^{1-a_{ij}} (-1)^n B_{i,\overline{a_{ij}}+\overline{p_i}}^{(n)}B_j B_{i, \ov{p}_i}^{(1-a_{ij}-n)}=0.
\end{align*}
Moreover, $\Ui$ admits a $\K(q)$-algebra anti-involution $\sigma$ which sends $B_i\mapsto B_i$ for all $i$.
\end{thm}

\begin{proof}
Follows from Theorem \ref{thm:Serre} by noting that $Y^\bi=\emptyset$ and $\tau i=i$ for all $i\in I$.
\end{proof}

\subsection{Change of parameters}

By \cite[Theorem 7.1]{Ko14} (also cf. Remark~\ref{rem:kappa}),
a presentation of the $\K(q)$-algebra $\Ui_{\vs,\kappa}$ is independent of the parameters $\kappa_i$. It is also well known that the $\K(q)$-algebra $\Ui_{\vs,\kappa}$ (up to some field extension) is isomorphic to $\Ui_{\vs^\diamond,{\bf0}}$ for some distinguished parameters $\vs^\diamond$, i.e., $\vs^\diamond=q_i^{-1}$ for all $i\in I$ such that $\tau i=i$ (cf. \cite{Le02}, \cite[Proposition~ 9.2]{Ko14}). Let us formulate this precisely for later use.

For given parameters $\vs$ satisfy \eqref{vs=}, let $\vs^\diamond$ be the associated distinguished parameters such that $\vs_i^\diamond=\vs_i$ if $\tau i\neq i$, and
\begin{equation}
\label{eq:par0}
\vs^\diamond_i=q_i^{-1}, \text{ if } \tau i=i.
\end{equation}
Let $\Ui_{\vs^\diamond,{\bf0}}$ be the $\imath${}quantum group with the parameters $\vs^\diamond$ and $\kappa_i=0$ for all $i\in I$.
Let ${\F}= \K(q)(a_i\mid i\in I\text{ such that } \tau i=i)$ be a field extension of $\K(q)$, where
\begin{align}
a_i=\sqrt{q_i\vs_i}, \qquad \forall i\in I \text{ such that }\tau i=i.
\end{align}
Denote by $_{\F}\Ui_{\vs,\kappa} =\F \otimes_{\K(q)} \Ui_{\vs,\kappa}$ the $\F$-algebra obtained by a base change.

\begin{prop}
  \label{prop:morphism}
There exists an isomorphism of ${\mathbb F}$-algebras
\begin{align*}
\phi_\imath: {}_{\F}\Ui_{\vs^\diamond,{\bf0}} & \longrightarrow {}_{\F}\Ui_{\vs,\kappa},
\\
B_i \mapsto \left\{\begin{array}{ll} B_i, & \text{ if }\tau i\neq i, \\  a_i^{-1} B_i, & \text{ if } \tau i =i;\end{array} \right.
\qquad K_\mu & \mapsto K_\mu, \quad (\forall i\in I, \mu\in Y^\bi),
\end{align*}
\end{prop}

\begin{proof}
Note that $B_i$ in $\Ui_{\vs^\diamond,\bf0}$ and $\Ui_{\vs,\kappa}$ have different expressions under their respective embeddings into $\U$.

We consider the following (rescaling) automorphism of the $\F$-algebra $_{\F}\U :=\F\otimes_{\K(q)} \U$ such that
\begin{align}
  \label{eqn: definition of phi}
  \phi_{\texttt u} : {}_{\F}\bU & \longrightarrow {}_{\F}\bU,
  \\
E_i  \mapsto \left\{\begin{array}{ll} E_i, & \text{ if }\tau i\neq i, \\  a_i E_i, & \text{ if } \tau i =i, \end{array} \right.
 \quad & F_i  \mapsto \left\{\begin{array}{ll} F_i, & \text{ if }\tau i\neq i, \\  a_i^{-1} F_i, & \text{ if }\tau i =i, \end{array} \right.
%F_i & \mapsto a_i^{-1} F_i,
\quad
K_\mu\mapsto K_\mu \quad (\forall  i\in I, \mu \in Y).
\notag
\end{align}

A direct computation shows that the automorphism $\phi_{\texttt u}$ on $_{\F}\U$ restricts to an ${\mathbb F}$-algebra isomorphism
\begin{align*}
\phi_\imath: {}_{\F}\Ui_{\vs^\diamond,{\bf0}} & \longrightarrow {}_{\F}\Ui_{\vs,\bf0},
\\
B_i \mapsto \left\{\begin{array}{ll} B_i, & \text{ if }\tau i\neq i, \\  a_i^{-1} B_i, & \text{ if } \tau i =i;\end{array} \right.
\qquad K_\mu & \mapsto K_\mu, \quad (\forall i\in I, \mu\in Y^\bi).
\end{align*}
%by noting that $a_i=a_{\tau i}$ for any $i\in I$ such that $\tau i=i$.
%, since $\vs_{\tau i}=\vs_i$ in this case.% for any $i\in I$ such that $\tau i=i$.

By \cite[Theorem 7.1]{Ko14}, there is an ${\mathbb F}$-algebra isomorphism
${}_{\F}\Ui_{\vs,{\bf0}} \stackrel{\cong}{\longrightarrow} {}_{\F}\Ui_{\vs,\kappa}$ which matches the corresponding generators
$B_i, K_\mu$.

The isomorphism in the proposition follows by composing these two isomorphisms.
\end{proof}

\subsection{Reduction to a $q$-binomial identity}

For
\begin{equation}
  \label{eq:wul}
  w\in\Z, \quad u,\ell\in\Z_{\geq0}, \text{ with } u,\ell \text{ not both } 0,
\end{equation}
we define
\begin{align}\label{eq:Twux}
T& (w,u,\ell)  \\
&= \sum_{\substack{c,e,r\geq0 \\ c+e+r=u}}
\sum^{\ell}_{\substack{t=0 \\ 2\mid(t+w-r) }}
\notag\\
&\quad q^{-t(\ell+u-1)+(\ell+u)(c-e)}
\qbinom{\ell}{t}_{q}\qbinom{w+t-\ell}{r}_{q} \qbinom{u-1+\frac{w+t-r}{2}}{c}_{q^2}\qbinom{\frac{w+t-r}{2}-\ell}{e}_{q^2}
  \notag \\
&-\sum_{\substack{c,e,r\geq0 \\ c+e+r=u}}
\sum^{\ell}_{\substack{t=0 \\ 2\nmid(t+w-r) }} \notag \\ \notag
&\quad q^{-t(\ell+u-1)+(\ell+u-1)(c-e)}
 \qbinom{\ell}{t}_{q}\qbinom{w+t-\ell}{r}_{q} \qbinom{u+\frac{w+t-r-1}{2}}{c}_{q^2}\qbinom{\frac{w+t-r-1}{2}-\ell}{e}_{q^2}.
\end{align}
%When there is no confusion, we may drop the index to write $T(w,u,\ell)$.

\begin{prop}\label{prop:Serre=T}
If $T(w,u,\ell)=0$ for all integers $w, u, \ell$ as in \eqref{eq:wul}, then the $\imath${}Serre relations \eqref{relation2} hold in the $\imath${}quantum group $\Ui_{\vs^\diamond,{\bf0}}$.
\end{prop}
The proof of Proposition \ref{prop:Serre=T} will be given in Section \ref{sec:Serre=T} and Appendix~\ref{sec:appendix}.

\subsection{A $q$-binomial identity}

The following is another main result of this paper, which will be generalized and proved in Section~ \ref{sec:T=0}.
\begin{thm}\label{thm:T=0}
The identity $T(w,u,\ell)=0$ holds, for all integers $w, u, \ell$ as in \eqref{eq:wul}.
%for any $w\in\Z$ and $u,\ell\in\Z_{\geq0}$, with $u,\ell$ not both $0$.
\end{thm}

\subsection{Proof of Theorem \ref{thm:Serre}}
  \label{subsec:proofSerre}

We assume the validity of Proposition \ref{prop:Serre=T} and Theorem \ref{thm:T=0}.

First, we consider the $\imath${}quantum group with distinguished parameters $\vs^\diamond$, $\Ui_{\vs^\diamond,{\bf0}}$. By the earlier works \cite{Le02, Ko14, BK15} as explained in Remark~\ref{rem:kappa},  it remains to prove the $\imath${}Serre relations \eqref{relation2} (with distinguished parameters $\vs^\diamond$). Indeed, the $\imath${}Serre relations  \eqref{relation2} follow by combining Proposition~\ref{prop:Serre=T} and Theorem \ref{thm:T=0}.

By a direct computation, the $\imath${}Serre relations for $\Ui_{\vs^\diamond,{\bf0}}$ is transformed into the $\imath${}Serre relations \eqref{relation2} for $\Ui_{\vs,\kappa}$ with general parameters by the isomorphism $\phi$ in Proposition~\ref{prop:morphism}. Applying Remark~\ref{rem:kappa} again, we have completed the proof of Theorem \ref{thm:Serre}.

\qed

\subsection{Bar involution on $\Ui$}
  \label{subsec:bar}

%In this subsection, we consider the existence of the bar involution for the quasi-split $\imath${}quantum group $\Ui:=\Ui_{\vs,\kappa}$.

%In \cite{BW13,BK15}, it is proved that for a suitable choice of parameter $\vs$, there exists a $\K$-algebra automorphism $^{\ov{\,\,\,\,\,}}: \Ui\rightarrow \Ui, \,\, x\mapsto \bar{x}$ which coincides with the usual bar involution for $\bU$, which satisfies
%$$\overline{B_i}=B_i,\text{ for all }i\in I.$$

%Recall that the \emph{bar involution} for quantum group $\U$ is the $\K$-algebra automorphism $^{\ov{\,\,\,\,\,}}: \U\rightarrow\U$, $x\mapsto \bar{x}$ defined by
%\begin{align}
%\ov{q}=q^{-1},\quad \ov{E_i}=E_i,\quad \ov{F_i}=F_i,\quad \ov{K_\mu}=K_\mu^{-1}.
%\end{align}

%Define $\K_0(q)=\{\la\in \K(q)\mid \ov{\la}=\la\}$, which is a subfield of $\K(q)$.

\begin{prop}
   \label{prop:bar}
 Assume the parameters $\vs_i$, for $i\in I$, satisfy the conditions (a)-(c):
\begin{enumerate}
\item[(a)] $\ov{\vs_iq_i} =\vs_iq_i$, if $\tau i=i$ and $a_{ij}\neq 0$ for some $j\in I\setminus\{i\}$;
 %$\ov{\vs_iq_i} =\vs_iq_i$, if $\tau i=i$;
\item[(b)] $\ov{\vs_i} =\vs_i =\vs_{\tau i}$, if $\tau i\neq i$ and $a_{i,\tau i}=0$;
\item[(c)] $\vs_{\tau i}=q_i^{-a_{i,\tau i}}\ov{\vs_i}$, if $\tau i\neq i$ and $a_{i,\tau i}\neq0$.
\end{enumerate}
Then there exists a $\K$-algebra automorphism $ ^{\ov{\,\,\,\,\,}}: \Ui\rightarrow \Ui$ (called a bar involution) such that
\[
\ov{q}=q^{-1}, \quad
\ov{K_\mu}=K_{\mu}^{-1}, \quad
\ov{B_i}=B_i,  \quad
\forall \mu\in Y^\bi, i\in I.
\]
\end{prop}

\begin{proof}
Under the assumptions, the $\imath${}divided powers $B_i^{(n)}$ in \eqref{eq:iDP1} and $B_{i, {\ov{p}}}^{(n)}$, for $\ov{p} \in \Z_2$, in \eqref{eq:iDPodd}-\eqref{eq:iDPev} are clearly bar invariant. If follows by inspection that all the explicit defining relations for $\Ui$ in \eqref{relation1}-\eqref{relation2} are bar invariant.
\end{proof}

\begin{rem}
%Proposition~\ref{prop:bar} remains valid if one replaces Condition (a) by a slightly weaker one:
%\begin{enumerate}
%\item[(a$'$)] $\ov{\vs_iq_i} =\vs_iq_i$, if $\tau i=i$ and $a_{ij}\neq 0$ for some $j\in I\setminus\{i\}$.
%\end{enumerate}
One could further check that Conditions (a)-(c) in Proposition~\ref{prop:bar} are necessary for the existence of the bar involution as well.
%However, under (a$'$) the $\imath${}divided powers $B_i^{(n)}$ may not be guaranteed to be bar invariant. The formulation of Proposition~\ref{prop:bar} makes it easier for future references.
%\end{rem}
%
%\begin{rem}
Under the constraint \eqref{eq:aij} on the Cartan matrix $A=(a_{ij})$, Proposition \ref{prop:bar} and the necessity of the conditions on parameters were known earlier in \cite{BK15}.
\end{rem}

%%%%%%%%%%%
%%%%%%%%%%%
\section{Reduction of $\imath${}Serre relations to a $q$-identity}
   \label{sec:Serre=T}

This section is devoted to a proof of Proposition \ref{prop:Serre=T}. As observed above, by the isomorphism $\phi$ in Proposition \ref{prop:morphism}, the $\imath${}Serre relations for $\Ui_{\vs^\diamond,{\bf0}}$ with distinguished parameters $\vs^\diamond$ is transformed into the $\imath${}Serre relations \eqref{relation2} for $\Ui_{\vs,\kappa}$ with general parameters. Hence we can and shall work the $\imath${}quantum groups with distinguished parameters $\vs_i^\diamond$, $\Ui=\Ui_{\vs^\diamond,{\bf0}}$, in this section on reduction of the $\imath${}Serre relations.

\subsection{Reduction by equivalence}
 \label{subsec:equiv}

%Then we have the following lemma.
\begin{lem}
  \label{lem: E equivalent F}
For any $i\in I$ such that $\tau i=i$ and each $\ov{p}\in\Z_2$,
%if \eqref{eqn:general F} holds, then so is \eqref{eqn:general E}.
then the following 2 identities in $\U$ are equivalent: for $j \neq i\in I$,
\begin{align}
\sum_{n=0}^{1-a_{ij}} (-1)^n  B_{i,\ov{a_{ij}}+\ov{p}}^{(n)}F_jB_{i,\ov{p}}^{(1-a_{ij}-n)} &=0,
\label{eqn:general F} \\
\sum_{n=0}^{1-a_{ij}} (-1)^n  B_{i, \ov{a_{ij}}+\ov{p}}^{(n)} E_{\tau j}\tK^{-1}_j B_{i,\ov{p} }^{(1-a_{ij}-n)} &=0.
\label{eqn:general E}
\end{align}
\end{lem}

\begin{proof}
Recall the involution $\varpi$  from Lemma~\ref{lemma:involution of U}  and the involution $\tau$ of $\U$ from \eqref{eq:tau}.
Assume the identity \eqref{eqn:general F} holds.
By definition, we have $\tau\circ \varpi (F_i+q_i^{-1} E_{i} \tK_i^{-1})= B_i$ as $\tau i=i$. It then follows by definition of the $\imath${}divided powers \eqref{eq:iDPodd}--\eqref{eq:iDPev} that $\tau\circ \varpi(B_i^{(n)}|_{\ov{p}} )=B_i^{(n)}|_{\ov{p}}$ for any $n, \ov{p}$. Hence applying $\tau\circ\varpi$ to \eqref{eqn:general F} gives us
\begin{align*}
\sum_{n=0}^{1-a_{ij}} (-1)^n  B^{(n)}_{i,\ov{a_{ij}}+\ov{p}} E_{\tau j}\tK^{-1}_{\tau j}B^{(1-a_{ij}-n)}_{i,\ov{p}}=0.
\end{align*}
Then  \eqref{eqn:general E} follows by multiplying the above identity on the right by $\tK_{\tau j}\tK_j^{-1}$ and noting that $\tK_{\tau j}\tK_j^{-1} B_i=B_i\tK_{\tau j}\tK_j^{-1}$.

Similarly by applying $\tau\circ\varpi$ to the identity \eqref{eqn:general E}, we can show that \eqref{eqn:general E} implies \eqref{eqn:general F}.
\end{proof}

Since we have $B_j=F_j +\vs_j E_{\tau j}\tK_j^{-1}$, the $\imath${}Serre relation \eqref{relation2} in $\Ui$ follows from the identities \eqref{eqn:general F}--\eqref{eqn:general E}, and it suffices to prove \eqref{eqn:general F} by Lemma~\ref{lem: E equivalent F}.

As \eqref{eqn:general F} is a statement in a rank 2 quantum group, for simplicity of notations, we further set  $i=1$ and $j=2$ in the remainder of this section.

\begin{rem}
  \label{rem:equiv+}
 When we deal with general $\imath${}quantum groups as in Remark~\ref{rem:split+}, a variant of Lemma~\ref{lem: E equivalent F} remains valid when we use a variant of the identity \eqref{eqn:general E} where $E_{\tau j}$ is replaced by $\texttt{T}_{w_\bullet} (E_{\tau j})$. Hence in this case, the $\imath${}Serre relation \eqref{relation2} in $\Ui$ follows again from the identity \eqref{eqn:general F} alone (which we establish in this paper).
\end{rem}

\subsection{Expansion formula of the $\imath${}divided powers}

Fix $i=1$ and $j=2$. Recall from \eqref{eq:1star} the notation for the idempotents $\onestar_\la$, for $\la\in\Z$.
The following expansion formulas will play a crucial role in proving \eqref{eqn:general F}.

\begin{lem}
 \cite[Propositions 2.8, 3.5]{BeW18}
   \label{lem:iDPdot}
For $m\ge 1$ and $\la \in \Z$, we have
\begin{align}
%\displaybreak
B_{1,\ev}^{(2m)} \onestar_{2\la}
&\small
= \sum_{c=0}^m \sum_{a=0}^{2m-2c} q_1^{2(a+c)(m-a-\la)-2ac-\binom{2c+1}{2}} \qbinom{m-c-a-\la}{c}_{q_1^2}
 E^{(a)}_1  F^{(2m-2c-a)}_1\onestar_{2\la},
\label{t2mdot}
\\
B_{1,\ev}^{(2m-1)} \onestar_{2\la}
&=  \sum_{c=0}^{m-1} \sum_{a=0}^{2m-1-2c}
\label{t2m-1dot}\\
&
\quad q_1^{2(a+c)(m-a-\la)-2ac-a-\binom{2c+1}{2}}  \qbinom{m-c-a-\la-1}{c}_{q_1^2}  E_1^{(a)}  F_1^{(2m-1-2c-a)}\onestar_{2\la}.
\notag
\\
B_{1,\odd}^{(2m)} \onestar_{2\la-1}
&= \sum_{c=0}^m \sum_{a=0}^{2m-2c}
\label{t2mdot2}  \\
&
\quad q_1^{2(a+c)(m-a-\la)-2ac+a-\binom{2c}{2}} \qbinom{m-c-a-\la}{c}_{q_1^2}
 E^{(a)}_1  F^{(2m-2c-a)}_1\onestar_{2\la-1},
\notag
\\
B_{1,\odd}^{(2m+1)} \onestar_{2\la-1}
&=  \sum_{c=0}^{m} \sum_{a=0}^{2m+1-2c}
\label{t2m-1dot2} \\
&
\small \quad q_1^{2(a+c)(m-a-\la)-2ac+2a-\binom{2c}{2}}  \qbinom{m-c-a-\la+1}{c}_{q_1^2}  E_1^{(a)}  F_1^{(2m+1-2c-a)}\onestar_{2\la-1}.
\notag
\end{align}
In particular, we have $B_{1,\ev}^{(n)}\onestar_{2\la}\in\,_{\A}\Udot_{1,\ev}$ and
$B_{1,\odd}^{(n)}\onestar_{2\la-1}\in \,_{\A}\Udot_{1,\odd}$ for all $n\in \N$.
\end{lem}

%%\subsection{Separation into 4 cases}

%Recall $\tau i=i$.
The identity \eqref{eqn:general F} is equivalent to the following 4 relations \eqref{eq:serre11F}--\eqref{eq:serre11evenodd}:
\begin{align}
%\displaybreak
\sum_{n=0}^{2m+1} (-1)^n  B_{1,\ev}^{(n)}F_2 B_{1,\ev}^{(2m+1-n)}=0, &\quad \text{ if }-a_{12}=2m \in 2\N;
  \label{eq:serre11F} \\
\sum_{n=0}^{2m+1} (-1)^n  B_{1,\odd}^{(n)}F_2 B_{1,\odd}^{(2m+1-n)}=0, &\quad \text{ if }-a_{12}=2m \in 2\N;
  \label{eq:serre11odd}\\
\sum_{n=0}^{2m} (-1)^n  B_{1,\odd}^{(n)}F_2 B_{1,\ev}^{(2m-n)}=0, &\quad \text{ if }-a_{12}=2m-1 \in2\N+1;
  \label{eq:serre11oddeven}\\
\sum_{n=0}^{2m} (-1)^n  B_{1,\ev}^{(n)}F_2 B_{1,\odd}^{(2m-n)}=0, &\quad \text{ if }-a_{12}=2m-1 \in 2\N+1.
  \label{eq:serre11evenodd}
\end{align}
The necessity of applying different formulas in Lemma~\ref{lem:iDPdot} forces us to divide the proof of the identity \eqref{eqn:general F} into the 4 cases \eqref{eq:serre11F}--\eqref{eq:serre11evenodd}.

In the remainder of this section, we reduce the proof of the identity \eqref{eq:serre11F} to the $q$-binomial identity $T(w,u,\ell)=0$ in Theorem~\ref{thm:T=0};  similar reductions of the other relations \eqref{eq:serre11odd}--\eqref{eq:serre11evenodd} to the same identity are given in Appendix \ref{sec:appendix}.

\subsection{Computing $\imath${}Serre in $\Udot$}
   \label{subsection: QSP11F}

Let $a_{12}=-2m$. We shall use (\ref{t2mdot})--(\ref{t2m-1dot}) to rewrite the element
\begin{align}
  \label{eq:Serre-idem}
\sum_{n=0}^{2m+1} (-1)^n  B_{1,\ev}^{(n)}F_2 B_{1,\ev}^{(2m+1-n)}\onestar_{2\la} \in \Udot
\end{align}
for any $\la\in\Z$ in terms of monomial basis in $E_1, F_1, F_2$.

\vspace{2mm}
\noindent{\underline{Case I: $n$ is even}.} It follows from \eqref{t2m-1dot} that
\begin{align*}
B_{1,\ev}^{(2m+1-n)}\onestar_{2\la}
&=\sum_{c=0}^{m-\frac{n}{2}} \sum_{a=0}^{2m+1-n-2c}q_1^{(a+c)(2m+2-n-2a-2\la)-2ac-a-c(2c+1)} \\&
\qquad \cdot \qbinom{m-\frac{n}{2}-c-a-\la}{c}_{q_1^2}  E_1^{(a)}  F_1^{(2m+1-n-2c-a)}\onestar_{2\la}.
\notag
\end{align*}
Since $a_{12}=-2m$, by \eqref{eqn: idempotent Ei Fi}--\eqref{eqn: idempotent Ej Fj} we have $F_2\onestar_{\la}=\onestar_{\la+2m}F_2$, and hence
\begin{align*}
F_2 E_1^{(a)}  F_1^{(2m+1-n-2c-a)}\onestar_{2\la}
%&=\onestar_{2\la+2a+2m-2(2m+1-n-2c-a)}F_2 E_1^{(a)}  F_1^{(2m+1-n-2c-a)}\\
&=\onestar_{2(\la+2a-m-1+n+2c)}F_2 E_1^{(a)}  F_1^{(2m+1-n-2c-a)}.
\end{align*}
Furthermore, by using (\ref{t2mdot}), we have
\begin{align*}
 %\displaybreak
B_{1,\ev}^{(n)} & \onestar_{2(\la+2a-m-1+n+2c)}\\
&=\sum_{e=0}^{\frac{n}{2}} \sum_{d=0}^{n-2e} q_1^{2(d+e)(\frac{n}{2}-d-\la-2a+m+1-n-2c)-2de-e(2e+1)} \\
 &\qquad \cdot \qbinom{\frac{n}{2}-e-d-\la-2a-n-2c+m+1}{e}_{q_1^2}  E_1^{(d)}  F_1^{(n-2e-d)}\onestar_{2(\la+2a-m-1+n+2c)}
\\ &=\sum_{e=0}^{\frac{n}{2}} \sum_{d=0}^{n-2e}
 q_1^{2(d+e)(m+1-d-\la-2a-\frac{n}{2}-2c)-2de-e(2e+1)} \\
 &\qquad \cdot \qbinom{m+1-e-d-\la-2a-\frac{n}{2}-2c}{e}_{q_1^2}  E_1^{(d)}  F_1^{(n-2e-d)}\onestar_{2(\la+2a+n+2c-m-1)}.
\end{align*}
Hence combining the above 3 computations gives us
\begin{align}
  \label{eq:BFB}
 % \displaybreak
B_{1,\ev}^{(n)} & F_2 B_{1,\ev}^{(2m+1-n)}\onestar_{2\la} \\
 &=\sum_{e=0}^{\frac{n}{2}} \sum_{d=0}^{n-2e}\sum_{c=0}^{m-\frac{n}{2}} \sum_{a=0}^{2m+1-n-2c}
q_1^{(a+c+d+e)(2m+1-n-2\la-2a-2c-2d-2e)+d}
  \notag \\
 & \quad \cdot \qbinom{m+1-e-d-\la-2a-\frac{n}{2}-2c}{e}_{q_1^2} \qbinom{m-\frac{n}{2}-c-a-\la}{c}_{q_1^2}
  \notag \\
 & \quad \cdot E_1^{(d)}  F_1^{(n-2e-d)}F_2 E_1^{(a)}  F_1^{(2m+1-n-2c-a)}\onestar_{2\la}.
 \notag
\end{align}
Next, we move the divided powers of $E_1$ in the middle to the left.
Using \eqref{eqn: idempotent Ei Fi}--\eqref{eqn: idempotent Ej Fj} we have
\begin{align*}
%\displaybreak
F_2F_1^{(2m+1-n-2c-a)}\onestar_{2\la} %&=\onestar_{2\la+2m-2(2m+1-n-2c-a)}F_2  F_1^{(2m+1-n-2c-a)}\\
 &=\onestar_{2(\la+n+2c+a-m-1)}F_2  F_1^{(2m+1-n-2c-a)}   ;
\end{align*}
Using \eqref{eqn:commutate-idempotent3} we have
\begin{align*}
%\displaybreak
F_1^{(n-2e-d)} & E_1^{(a)} \onestar_{2(\la+n+2c+a-m-1)}\\
&=\sum^{\min\{a,n-2e-d\}}_{r=0}\qbinom{n-2e-d-a-2(\la+n+2c+a-m-1)}{r}_{q_1}\\
&\qquad\qquad\qquad\quad \cdot E_1^{(a-r)}F_1^{(n-2e-d-r)}\onestar_{2(\la+n+2c+a-m-1)}\\
&=\sum^{\min\{a,n-2e-d\}}_{r=0}\qbinom{2m+2-2e-d-3a-2\la-4c-n}{r}_{q_1} \\
&\qquad\qquad\qquad\quad \cdot E_1^{(a-r)} F_1^{(n-2e-d-r)}\onestar_{2(\la+n+2c+a-m-1)}.
\end{align*}

Plugging these new formulas into \eqref{eq:BFB}, we obtain
\begin{align}
%\displaybreak
&  \sum_{n=0,2\mid n}^{2m+1}
    B_{1,\ev}^{(n)}F_2 B_{1,\ev}^{(2m+1-n)}\onestar_{2\la}\label{eqn: first even}\\
&=\sum_{n=0,2\mid n}^{2m+1}\sum_{c=0}^{m-\frac{n}{2}}\sum_{e=0}^{\frac{n}{2}} \sum_{a=0}^{2m+1-n-2c}\sum_{d=0}^{n-2e}\sum^{\min\{a,n-2e-d\}}_{r=0}q_1^{(a+c+d+e)(2m+1-n-2\la-2a-2c-2d-2e)+d}\notag\\
&\quad\cdot \qbinom{a+d-r}{d}_{q_1}
\qbinom{2m+2-2e-d-3a-2\la-4c-n}{r}_{q_1}
\qbinom{m-\frac{n}{2}-c-a-\la}{c}_{q_1^2} \notag \\
& \quad \cdot
\qbinom{m+1-e-d-\la-2a-\frac{n}{2}-2c}{e}_{q_1^2}
E_1^{(a+d-r)}F_1^{(n-2e-d-r)}F_2F_1^{(2m+1-n-2c-a)}\onestar_{2\la}.
\notag
\end{align}

\vspace{2mm}
\noindent{\underline{Case II: $n$ is odd}.}
Similarly, by (\ref{t2mdot}) we have
\begin{align*}
B_{1,\ev}^{(2m+1-n)}\onestar_{2\la}&=\sum_{c=0}^{m+\frac{1-n}{2}} \sum_{a=0}^{2m+1-n-2c} q_1^{(a+c)(2m+1-n-2a-2\la)-2ac-c(2c+1)} \\&
\qquad \cdot \qbinom{m+\frac{1-n}{2}-c-a-\la}{c}_{q_1^2}  E_1^{(a)}  F_1^{(2m+1-n-2c-a)}\onestar_{2\la}.
\notag
\end{align*}
%\begin{align*}
%F_2 E_1^{(a)}  F_1^{(2m+1-i-2c-a)}\onestar_{2\la}&=\onestar_{2\la+2a+2m-2(2m+1-i-2c-a)}F_2 E_1^{(a)}  F_1^{(2m+1-i-2c-a)}\\
                                            % &=\onestar_{2(\la+2a-m-1+i+2c)}F_2 E_1^{(a)}  F_1^{(2m+1-i-2c-a)}   .
%\end{align*}
Using (\ref{t2m-1dot}) we have
\begin{align*}
 %\displaybreak
B_{1,\ev}^{(n)} & \onestar_{2(\la+2a-m-1+n+2c)}\\
&=\sum_{e=0}^{\frac{n-1}{2}} \sum_{d=0}^{n-2e} q_1^{2(d+e)(\frac{n+1}{2}-d-\la-2a+m+1-n-2c)-2de-d-e(2e+1)} \\
&\quad \cdot \qbinom{\frac{n+1}{2}-e-d-\la-2a-n-2c+m}{e}_{q_1^2}  E_1^{(d)}  F_1^{(n-2e-d)}\onestar_{2(\la+2a-m-1+n+2c)}
\\ &=\sum_{e=0}^{\frac{n-1}{2}} \sum_{d=0}^{n-2e}q_1^{2(d+e)(m+\frac{3}{2}-d-\la-2a-\frac{n}{2}-2c)-2de-d-e(2e+1)} \\
&\quad \cdot \qbinom{m+\frac{1}{2}-e-d-\la-2a-\frac{n}{2}-2c}{e}_{q_1^2}  E_1^{(d)}  F_1^{(n-2e-d)}\onestar_{2(\la+2a+n+2c-m-1)}.
\end{align*}
%\begin{align*}
%&q^{(a+c)(2m+1-i-2a-2\la)-2ac-c(2c+1)} q^{2(d+e)(m+\frac{3}{2}-d-\la-2a-\frac{i}{2}-2c)-2de-d-e(2e+1)}=\\
%&q^{(a+c+d+e)(2m+2-i-2\la-2a-2c-2d-2e)-a-2c}
%\end{align*}
Combining the above two formulas and simplifying the resulting expression, we obtain the following equality:
\begin{align}
%\displaybreak
& \sum_{n=1,2\nmid n}^{2m+1} B_{1,\ev}^{(n)}F_2 B_{1,\ev}^{(2m+1-n)}\onestar_{2\la}\label{eqn: first odd}\\
&  =  \sum_{n=1,2\nmid n}^{2m+1}\sum_{c=0}^{m+\frac{1-n}{2}} \sum_{e=0}^{\frac{n-1}{2}}
\sum_{a=0}^{2m+1-n-2c}\sum_{d=0}^{n-2e}\sum^{\min\{a,n-2e-d\}}_{r=0}
q_1^{(a+c+d+e)(2m+2-n-2\la-2a-2c-2d-2e)-a-2c} \notag \\
&  \qquad \cdot \qbinom{a+d-r}{d}_{q_1}
\qbinom{2m+2-2e-d-3a-2\la-4c-n}{r}_{q_1}
\qbinom{m+\frac{1-n}{2}-c-a-\la}{c}_{q_1^2}   \notag\\
& \qquad \cdot
 \qbinom{m+\frac{1}{2}-e-d-\la-2a-\frac{n}{2}-2c}{e}_{q_1^2}
 E_1^{(a+d-r)}F_1^{(n-2e-d-r)}F_2F_1^{(2m+1-n-2c-a)}\onestar_{2\la}.
\notag
\end{align}

Therefore, by combining the computations \eqref{eqn: first even}--\eqref{eqn: first odd} which depend on the parity of $n$ above, we obtain the following formula for \eqref{eq:Serre-idem}:
\begin{align}\label{eq:evev}
 %\displaybreak
\small \sum_{n=0}^{2m+1} (-1)^n & B_{1,\ev}^{(n)}F_2 B_{1,\ev}^{(2m+1-n)}\onestar_{2\la}= \\ \notag
&\small \sum_{n=0,2\mid n}^{2m}\sum_{c=0}^{m-\frac{n}{2}}\sum_{e=0}^{\frac{n}{2}} \sum_{a=0}^{2m+1-n-2c}\sum_{d=0}^{n-2e}\sum^{\min\{a,n-2e-d\}}_{r=0}q_1^{(a+c+d+e)(2m+1-n-2\la-2a-2c-2d-2e)+d} \notag \\
&\small \cdot \qbinom{a+d-r}{d}_{q_1}\qbinom{2m+2-n-2\la-2e-d-3a-4c}{r}_{q_1} \qbinom{m-\frac{n}{2}-\la-c-a}{c}_{q_1^2}
\notag\\
&\small \cdot \qbinom{m+1-\frac{n}{2}-\la-e-d-2a-2c}{e}_{q_1^2}E_1^{(a+d-r)}F_1^{(n-2e-d-r)}F_2F_1^{(2m+1-n-2c-a)}\onestar_{2\la} \notag \\
-&\small \sum_{n=1,2\nmid n}^{2m+1}\sum_{c=0}^{m+\frac{1-n}{2}}\sum_{e=0}^{\frac{n-1}{2}} \sum_{a=0}^{2m+1-n-2c}\sum_{d=0}^{n-2e}\sum^{\min\{a,n-2e-d\}}_{r=0}q_1^{(a+c+d+e)(2m+2-n-2\la-2a-2c-2d-2e)-a-2c} \notag \\
&\small \cdot \qbinom{a+d-r}{d}_{q_1}\qbinom{2m+2-n-2\la-2e-d-3a-4c}{r}_{q_1} \qbinom{m+\frac{1-n}{2}-\la-c-a}{c}_{q_1^2}
\notag\\
&\small \cdot \qbinom{m+\frac{1-n}{2}-\la-e-d-2a-2c}{e}_{q_1^2}E_1^{(a+d-r)}F_1^{(n-2e-d-r)}F_2F_1^{(2m+1-n-2c-a)}\onestar_{2\la}. \notag
\end{align}

%Let $a+d-j=x$, $i-2e-d-j=y$, $2m+1-i-2c-a=2m+1-x-y-2u$, then we have
%$a=x+y+2e+2j-i$, $d=i-2e-j-y$, $c=u-e-j$. $i,j$ and $e$ are considered as free variables.
%Further,

%\begin{align*}
%&a+c+d+e=x+u; \\
%&(a+c+d+e)(2m+1-i-2\la-2a-2c-2d-2e)+d\\
%&=(x+u)(2m+1-2\la-2x-2u)-i(x+u)+i-j-2e-y;\\
%&(a+c+d+e)(2m+2-i-2\la-2a-2c-2d-2e)-a-2c\\
%&=(x+u)(2m+1-2\la-2x-2u)-i(x+u)+i-y-u;\\
%&2m+2-i-2\la-2e-d-3a-4c=2m+2-2\la-4u-3x-2y-2e+i-j;\\
%&m-\frac{i}{2}-\la-c-a=m-\la-u-x-y+\frac{i}{2}-e-j;\\
%&m+1-\frac{i}{2}-\la-e-d-2a-2c\\
%&=m+1-\la-2x-y-2u+\frac{i}{2}-e-j.
%\end{align*}
Observe the monomials on the right-hand side of the equation \eqref{eq:evev} above are of the form $E_1^{(\ell)}F_1^{(y)}F_2F^{(2m+1-\ell-y-2u)}_1\onestar_{2\la}$, for $\ell, y, u \in \N$, so let us change variables to allow us to collect the like terms together.
Set
\[
\ell =a+d-r, \qquad y =n-2e-d-r.
\]
Noting that $n+2c+a-\ell-y$ is even, we can write
\[
n+2c+a-\ell-y =2u, \qquad \text{ for some }u\in\Z.
\]
Then we have $2m+1-n-2c-a=2m+1-\ell-y-2u$,
$a=\ell+y+2u-2c-n$, $d=n+c-e-u-y$, and
\[
u=c+e+r.
\]

If $u=0$ and $\ell=0$, then $e=c=r=a=d=0$. In this case, collecting the corresponding monomials in \eqref{eq:evev} together gives us $\sum^{2m+1}_{n=0}(-1)^nF^{(n)}_1F_2F^{(2m+1-n)}_1\onestar_{2\la}$, which equals $0$ by the $q$-Serre relation \eqref{eq:serre4}.

Now assume $u, \ell \in \N$, not both $0$. Then the monomial $E_1^{(\ell)}F_1^{(y)}F_2F^{(2m+1-\ell-y-2u)}_1\onestar_{2\la}$ in \eqref{eq:evev} has coefficient given by $q_1^{(\ell+u)(2m+1-2\la-2\ell-3u-y)} S(y,u,\ell, \la)$, where
\begin{align}
%\displaybreak
S(y,u,\ell, \la)
&:=\sum^{2m}_{n=0,2\mid n}
\sum_{\substack{c,e,r\geq0 \\ c+e+r=u}}
q_1^{(u+y-n)(\ell+u-1)+c-e}  \label{eq:T(x,y,u;w)}
\\
&\quad\cdot \qbinom{\ell}{-u-y-e+c+n}_{q_1}\qbinom{2m+2-2\la-5u-3\ell-2y-e+c+n}{r}_{q_1}\notag
\\
& \quad\cdot \qbinom{m-\la-2u-\ell-y+c+\frac{n}{2}}{c}_{q_1^2}\qbinom{m+1-\la-2\ell-y-3u+\frac{n}{2}+c}{e}_{q_1^2} \notag
\\
\notag &-\sum^{2m+1}_{n=1,2\nmid n}\sum_{\substack{c,e,r\geq0 \\ c+e+r=u}}q_1^{(u+y-n)(\ell+u-1)}\\ \notag
&\quad\cdot \qbinom{\ell}{-u-y-e+c+n}_{q_1}\qbinom{2m+2-2\la-5u-3\ell-2y-e+c+n}{r}_{q_1}
\\
&\quad\cdot \qbinom{m-\la-2u-\ell-y+c+\frac{n+1}{2}}{c}_{q_1^2}\qbinom{m-\la-2\ell-y-3u+\frac{n+1}{2}+c}{e}_{q_1^2}.\ \ \   \notag
\end{align}
Rewriting the identity \eqref{eq:evev} using \eqref{eq:T(x,y,u;w)} and its preceding discussions, we have proved the following.

\begin{prop}
We have
\begin{align}
  \label{eq:BFB=S}
&\sum_{n=0}^{2m+1} (-1)^n  B_{1,\ev}^{(n)}F_2 B_{1,\ev}^{(2m+1-n)}\onestar_{2\la}\\
&\quad =
\sum_{\substack{ \ell,y,u\geq0;u+\ell>0 \\ \ell+y+2u\leq 2m+1}}
q_1^{(\ell+u)(2m+1-2\la-2\ell-3u-y)} S(y,u,\ell,\la)E_1^{(\ell)}F_1^{(y)}F_2F^{(2m+1-\ell-y-2u)}_1\onestar_{2\la}.
\notag
\end{align}
\end{prop}

\subsection{Proof of Proposition~\ref{prop:Serre=T}}

%\begin{proof} [Proof of Proposition~\ref{prop:Serre=T}]

Using new variables $t :=-u-y-e+c+n$ and $w :=2m+2-2\la-2\ell-4u-y$, we have
\begin{equation}
  \label{eq:S=T}
S(y,u,\ell, \la)=T(w,u,\ell)|_{q\mapsto q_1}
\end{equation}
by a direct calculation; cf. \eqref{eq:Twux} for notation $T(w,u,\ell)$ and \eqref{eq:T(x,y,u;w)} for notation $S(y,u,\ell, \la)$.

By assumption, $T(w,u,\ell)=0$ for any $w\in\Z$ and $u,\ell \in \N$ with $u,\ell$ not both $0$. Hence by \eqref{eq:S=T} we have $S(y,u,\ell, \la)=0$, and then by \eqref{eq:BFB=S} we obtain
\begin{align}
  \label{eq:se1=0}
\sum_{n=0}^{2m+1} (-1)^n  B_{1,\ev}^{(n)}F_2 B_{1,\ev}^{(2m+1-n)}\onestar_{2\la} =0,
\qquad \forall \la\in\Z.
\end{align}
Thanks to Remark~\ref{rem:u=0}, this proves the identity \eqref{eq:serre11F}.

Similar reductions of the identities \eqref{eq:serre11odd}--\eqref{eq:serre11evenodd} to the $q$-binomial identity $T(w,u,\ell)=0$ in Theorem~\ref{thm:T=0} can be found in Appendix \ref{sec:appendix}.

Being equivalent to the 4 identities \eqref{eq:serre11F}--\eqref{eq:serre11evenodd}, the identity \eqref{eqn:general F} follows. Then by the reduction in \S\ref{subsec:equiv}, the $\imath${}Serre relation \eqref{relation2} holds. Proposition~\ref{prop:Serre=T} is proved.
\qed

%%%%%%%%%
%%%%%%%%%
\section{A $q$-binomial identity and generalization}
   \label{sec:T=0}

The section is devoted to a proof of Theorem~ \ref{thm:T=0}. We will first generalize $T(w,u,\ell)$ to a function $G$ which involves several new variables, and then establish various recursive relations for $G$ to show some generalized identities involving $G$.

\subsection{Function $G$ and its recursions}

For $w, p_0, p_1, p_2\in \bbZ$ and $u,\ell  \in \bbZ_{\ge 0}$, we define
\begin{align}\label{eq:Twxyu}
G(w,u, & \ell;p_0,p_1,p_2):=(-1)^{w}q^{u^2-wu+\ell u}\\
&\cdot \left\{ \sum_{\substack{c,e,r\geq0 \\ c+e+r=u}}
\sum^{\ell }_{\substack{t=0 \\ 2\mid(t+w-r) }}
q^{-t(\ell +u-1)-u(c+e)+2c+rp_0+2cp_1+2ep_2} \right.  \notag\\ \notag
&\qquad \cdot \qbinom{\ell }{t}_q\qbinom{w+t+p_0}{r}_q \qbinom{\frac{w+t-r}{2}+p_1}{c}_{q^2}\qbinom{\frac{w+t-r}{2}+p_2}{e}_{q^2}\\  \notag
&\quad -\sum_{\substack{c,e,r\geq0 \\ c+e+r=u}}
\sum^{\ell }_{\substack{t=0 \\ 2\nmid(t+w-r) }}
q^{-t(\ell +u-1)-(u-1)(c+e)+rp_0+2cp_1+2ep_2}\\ \notag
& \left. \qquad \cdot \qbinom{\ell }{t}_q\qbinom{w+t+p_0}{r}_q \qbinom{1+\frac{w+t-r-1}{2}+p_1}{c}_{q^2}\qbinom{\frac{w+t-r-1}{2}+p_2}{e}_{q^2} \right\}.  \notag
\end{align}
The following relation between $T$ and $G$ follows by definitions in \eqref{eq:Twux} and \eqref{eq:Twxyu}:
\begin{equation}
  \label{eq:T=G}
T(w,u,\ell )=(-1)^wq^{wu-u^2}G(w,u,\ell ;-\ell ,u-1,-\ell ).
\end{equation}

\begin{lem}\label{lem:Gwuxp0p1p2}
For any $w,p_0,p_1,p_2,k\in \bbZ$ and $u,\ell  \in \bbZ_{\ge 0}$, we have the following recursive relations:
\begin{align}
 &G(w+1, u, \ell ; p_0, p_1, p_2) \label{eq:Gw+1}\\
 &=q^{-2u}G(w,u,\ell ;p_0,p_2,p_1+1)-q^{2p_0+\ell }G(w,u-1,\ell ;p_0,p_1,p_2); \notag\\
&G(w,u,\ell ;p_0,p_1+1,p_2)=G(w,u,\ell ;p_0,p_1,p_2)+q^{4p_1+\ell +4}G(w,u-1,\ell ;p_0,p_1,p_2); \label{eq:G1+1} \\
&G(w,u,\ell ;p_0,p_1,p_2+1)=G(w,u,\ell ;p_0,p_1,p_2)+q^{4p_2+\ell +2}G(w,u-1,\ell ;p_0,p_1,p_2);\label{eq:G2+1} \\
%&G(w,u,\ell +1;p_0,p_1,p_2)=q^{u}G(w,u,\ell ;p_0,p_1,p_2)-q^{-u-2\ell }G(w,u,\ell ;p_0+1,p_2,p_1+1); \label{eq:Gx+1p0}\\
&G(w,u,\ell +1; p_0,p_1,p_2)=q^{u}G(w,u,\ell ;p_0,p_1,p_2)-q^{u-2\ell }G(w+1,u,\ell ;p_0,p_1,p_2); \label{eq:Gx+1w}\\
&G(w,u,\ell ;p_0,p_1,p_2)=q^{4ku}G(w+2k,u,\ell ;p_0-2k,p_1-k,p_2-k); \label{eq:Gk}\\
&G(w+1,u,\ell ;p_0,p_1,p_2)=q^{-2u}G(w,u,\ell ;p_0+1,p_2,p_1+1).\label{eq:Godd}\\  \notag
\end{align}
\end{lem}
\begin{proof}
We provide a detailed argument for \eqref{eq:Gw+1}. Applying the $q$-binomial identity
\begin{align}  \label{eq:ident-}
\qbinom{m}{t}=q^{-t}\qbinom{m-1}{t}_q+q^{m-t}\qbinom{m-1}{t-1}_q
\end{align}
to the second $q$-binomial in each summand of $G(w+1,u,\ell ;p_0,p_1,p_2)$ (obtained from \eqref{eq:Twxyu} with $w\to w+1$), we have
\begin{align*}
G(w+1,u,\ell ;p_0,p_1,p_2)=S_1+S_2.
\end{align*}
Here
\begin{eqnarray*}
S_1:&=&(-1)^{w+1}q^{u^2-(w+1)u+\ell u}\\ \notag
 &&\left\{ \sum_{\substack{c,e,r\geq0 \\ c+e+r=u}}
 \sum^{\ell }_{\substack{t=0 \\ 2\nmid(t+w-r) }}
 q^{-t(\ell +u-1)-u(c+e)+2c+rp_0+2cp_1+2ep_2-r}  \right.  \\ \notag
&& \quad\cdot \qbinom{\ell }{t}_q\qbinom{w+t+p_0}{r}_q \qbinom{\frac{w+t-r-1}{2}+p_1+1}{c}_{q^2}\qbinom{1+\frac{w+t-r-1}{2}+p_2}{e}_{q^2}\\  \notag
&&\quad -\sum_{\substack{c,e,r\geq0 \\ c+e+r=u}}
\sum^{\ell }_{\substack{t=0 \\ 2\mid(t+w-r) }} q^{-t(\ell +u-1)-(u-1)(c+e)+rp_0+2cp_1+2ep_2-r}\\ \notag
&&\left. \quad \cdot \qbinom{\ell }{t}_q\qbinom{w+t+p_0}{r}_q \qbinom{\frac{w+t-r}{2}+p_1+1}{c}_{q^2}\qbinom{\frac{w+t-r}{2}+p_2}{e}_{q^2}  \right \},
\end{eqnarray*}
and
\begin{eqnarray*}
S_2&:=&(-1)^{w+1}q^{u^2-(w+1)u+\ell u}\\ \notag
 &=&
\left\{ \sum_{\substack{c,e,r\geq0 \\ c+e+r=u}}
\sum^{\ell }_{\substack{t=0 \\ 2\nmid(t+w-r) }}
q^{-t(\ell +u-1)-u(c+e)+2c+rp_0+2cp_1+2ep_2+w+1+t+p_0-r}   \right. \\ \notag
&&\quad \cdot \qbinom{\ell }{t}_q\qbinom{w+t+p_0}{r-1}_q \qbinom{\frac{w+t-(r-1)}{2}+p_1}{c}_{q^2}\qbinom{\frac{w+t-(r-1)}{2}+p_2}{e}_{q^2}\\  \notag
&&\quad -\sum_{\substack{c,e,r\geq0 \\ c+e+r=u}}
\sum^{\ell }_{\substack{t=0 \\ 2\mid(t+w-r) }}
q^{-t(\ell +u-1)-(u-1)(c+e)+rp_0+2cp_1+2ep_2+w+1+t+p_0-r}\\ \notag
&&\left.  \quad\cdot \qbinom{\ell }{t}_q\qbinom{w+t+p_0 }{r-1}_q \qbinom{1+\frac{w+t-(r-1)-1}{2}+p_1}{c}_{q^2}\qbinom{\frac{w+t-(r-1)-1}{2}+p_2}{e}_{q^2} \right \}.
\end{eqnarray*}
By permutating the variables $c$, $e$, we obtain
\begin{align*}
S_1=q^{-2u}G(w,u,\ell ;p_0,p_2,p_1+1).
\end{align*}
By a change of variables $r \mapsto r+1$,  we have
\begin{align*}
S_2=-q^{2p_0+\ell }G(w,u-1,\ell ;p_0,p_1,p_2).
\end{align*}
Then \eqref{eq:Gw+1} follows by summing up $S_1$ and $S_2$ above.

The recursions \eqref{eq:G1+1}--\eqref{eq:G2+1} are proved similarly to \eqref{eq:Gw+1}.

The identity \eqref{eq:Gx+1w} can be proved similarly by
the following $q$-binomial identity:
\begin{align}\label{eq:ident+}
\qbinom{m}{t}=q^{t}\qbinom{m-1}{t}_q+q^{-m+t}\qbinom{m-1}{t-1}_q.
\end{align}
Finally, \eqref{eq:Gk}-\eqref{eq:Godd} can be easily verified directly. %In fact, \eqref{eq:Godd} can also be obtained by using \eqref{eq:Gx+1p0} and  \eqref{eq:Gx+1w}.
\end{proof}

\subsection{Specializations of the function $G$}

For $p_1,p_2\in \bbZ$ and $u \in \bbZ_{\ge 0}$, we define
\begin{align}\label{Hup1p2}
H(u; p_1,p_2):=\sum_{\begin{array}{c}c,e\geq0\\c+e=u\end{array}}q^{2c+2cp_1+2ep_2}\qbinom{p_1}{c}_{q^2}\qbinom{p_2}{e}_{q^2},\\ \notag
\end{align}
Note that $H(u;p_1,p_2)= G(0, u, 0;0, p_1, p_2)$, a specialization of $G$ defined in \eqref{eq:Twxyu}.

\begin{lem}\label{lem:Hup1p2}
For any $p_1,p_2\in \bbZ$ and $u \in \bbZ_{> 0}$, we have

\begin{align}
H(u; p_2,p_1+1) &=q^{2u}(H(u; p_1,p_2)+H(u-1; p_1,p_2)); \label{eq:Hp2p1+1} \\
H(u;p_1+1,p_2)  &=H(u; p_1,p_2)+q^{4(p_1+1)}H(u-1; p_1,p_2); \label{eq:Hp1+1} \\
H(u; p_1,p_2+1) &=H(u; p_1,p_2)+q^{4p_2+2}H(u-1; p_1,p_2). \label{eq:Hp2+1}
\end{align}
\end{lem}

\begin{proof}
%The proof of \eqref{eq:Hp2p1+1}--\eqref{eq:Hp2+1} are similar to that of \eqref{eq:Gw+1}--\eqref{eq:Gx+1w},
Note that \eqref{eq:Hp2p1+1} follows from \eqref{eq:ident+}, and \eqref{eq:Hp1+1}--\eqref{eq:Hp2+1} follow from \eqref{eq:ident-}.
\end{proof}

Define
\begin{align}
G_0(w,u;p_0,p_1,p_2) &:=G(w,u,0;p_0,p_1,p_2);
  \label{eq:G0}  \\
 G_{00}(w,u;p_1,p_2) &:=G(w,u,0;0,p_1,p_2).  \label{eq:G00}
\end{align}

Observe by definition that $H(u;p_1,p_2)= G_{00}(0, u; p_1, p_2)$.

\begin{prop}\label{prop:G=H}
 For any $w,p_1,p_2\in \bbZ$ and $u \in \bbZ_{\ge 0}$, $G_{00}(w,u;p_1,p_2)$ is independent of $w$; that is,
$%\begin{align*}
G_{00}(w,u;p_1,p_2) =H(u;p_1,p_2).
$%\end{align*}
\end{prop}

\begin{proof}
First assume $w\ge 0$. We prove the identity by induction on $w$. For $w=0$, the identity follows by definitions.
Using \eqref{eq:Gw+1}, \eqref{eq:Hp2p1+1} and the induction hypothesis, we have by definition of $G_{00}$ in \eqref{eq:G00} that
\begin{align*}
G_{00}(w+1,u;p_1,p_2 )
    & =q^{-2u}G_{00}(w,u;p_2,p_1+1)-G_{00}(w,u-1;p_1,p_2)\\
    & =q^{-2u}H(u;p_2,p_1+1)-H(u-1;p_1,p_2)\\
    & =H(u;p_1,p_2).
\end{align*}
Viewing $p_1, p_2, u$ as fixed, we regard the identity in the proposition as an identity involving rational functions in 2 variables $q, q^w$. Since this identity holds for all $w\ge 0$, it must hold as a formal identity in the 2 variables, and hence as an identity in $q$, for arbitrary $w\in\Z$.
\end{proof}

The following corollary is immediate by setting $p_2=0$ in Proposition~\ref{prop:G=H}. Recall the definition of $G_{00}$ in \eqref{eq:G00}.

\begin{cor}
The identity $G_{00}(w,u;p_1,0)=H(u;p_1,0)$ holds; that is,
\begin{align}
   \label{eq:GWup10}
& (-1)^wq^{u^2-wu}
 \left\{ \sum_{\substack{c+e+r=u\\ 2\mid(w-r)}}
q^{-u(c+e)+2c+2cp_1}\qbinom{w}{r}\qbinom{\frac{w-r}{2}+p_1}{c}_{q^2}
\qbinom{\frac{w-r}{2}}{e}_{q^2} \right.
  \\
  &- \left.  \sum_{\substack{c+e+r=u\\ 2\nmid(w-r)}}
 q^{-(u-1)(c+e)+2cp_1}\qbinom{w}{r}\qbinom{1+\frac{w-r-1}{2}+p_1}{c}_{q^2}
\qbinom{\frac{w-r-1}{2}}{e}_{q^2} \right\}
 =  q^{2u+2up_1} & \qbinom{p_1}{u}_{q^2}.\notag
\end{align}
\end{cor}

Recall $T(w, u, \ell)$ from \eqref{eq:Twux}.
\begin{cor}
\label{cor:x=0}
We have $%\begin{equation}\label{eqn: x=0}
T(w,u,0)=0,$ for any $w\in\bbZ, u\in\Z_{>0}$.
\end{cor}

\begin{proof}
It follow by the identity \eqref{eq:T=G} that
\begin{equation*}
T(w,u,0) =(-1)^wq^{wu-u^2}G_{00}(w,u;u-1,0) = (-1)^wq^{wu +u^2} \qbinom{u-1}{u}_{q^2}=0,
\end{equation*}
where the second equality above uses  \eqref{eq:GWup10}.
\end{proof}

\subsection{A multi-variable identity}

We now prove the main result of this section.
\begin{thm}
  \label{prop:Gwuxp0p1p2p3}
The following identity holds, for any $w, p_0, p_1, p_2\in \bbZ$, $\ell  \in \bbZ_{>0}$ and $u\in \bbZ_{\geq 0}$:
\begin{align*}
G(w,u,\ell ; p_0,p_1,p_2)=0.
\end{align*}
\end{thm}

\begin{proof}
By the recursion \eqref{eq:Gx+1w} on $\ell$, it suffices to prove the desired identity at $\ell = 1$.
%, i.e, $G(w,u,1,p_0,p_1,p_2)=0$ for any $p_0,p_1,p_2\in \bbZ$.
The proof is divided into the following two cases. Recall the definitions of $G_{0}$ in \eqref{eq:G0} and $G_{00}$ in \eqref{eq:G00}.

\vspace{2mm}
\noindent{\underline{Case I: $p_0$ is even}.} From the recursive relations \eqref{eq:Gx+1w}--\eqref{eq:Godd} and   Proposition~ \ref{prop:G=H} we have:
\begin{align*}
G & (w,u,1; p_0,p_1,p_2) \\
&= q^uG_0(w, u; p_0,p_1,p_2) -q^{u}G_0(w+1,u; p_0,p_1,p_2)\\
&= q^{u+2p_0u} G_{00}(w+p_0, u ;p_1-\frac{p_0}{2}, p_2-\frac{p_0}{2})
\\
&\qquad -q^{u+2p_0u}G_{00}(w+p_0+1, u; p_1-\frac{p_0}{2},p_2-\frac{p_0}{2})\\
&= q^{u+2p_0u}H(u;p_1-\frac{p_0}{2},p_2-\frac{p_0}{2})-q^{u+2p_0u}H(u;p_1-\frac{p_0}{2},p_2-\frac{p_0}{2})
                     =0.
\end{align*}

\noindent{\underline{Case II: $p_0$ is odd}.} Similarly, we have
\begin{align*}
G & (w,u,1;p_0,p_1,p_2) \\
  &= q^uG_0(w,u;p_0, p_1, p_2)-q^{u}G_0(w+1,u;p_0, p_1, p_2)\\
  &= q^{u+2(p_0-1)u} G_{0}(w+p_0-1, u; 1, p_1-\frac{p_0-1}{2}, p_2-\frac{p_0-1}{2})\\
  &\qquad -q^{u+2(p_0-1)u} G_{0}(w+p_0, u; 1, p_1-\frac{p_0-1}{2},p_2-\frac{p_0-1}{2}) \\
  &=q^{2p_0u+u}G_{00}(w+p_0, u; p_2-\frac{p_0+1}{2},p_1-\frac{p_0-1}{2})\\
  &\qquad -q^{2p_0u+u}G_{00}(w+p_0+1, u; p_2-\frac{p_0+1}{2},p_1-\frac{p_0-1}{2})\\
 &=q^{2p_0u+u}(H(u;p_2-\frac{p_0+1}{2},p_1-\frac{p_0-1}{2})-H(u;p_2-\frac{p_0+1}{2},p_1-\frac{p_0-1}{2}))
  =0.
\end{align*}
The theorem is proved.
\end{proof}

\subsection{Proof of Theorem \ref{thm:T=0}}

Let $w\in \Z$, $u\in \N$. It follows from the identity \eqref{eq:T=G} and Theorem ~\ref{prop:Gwuxp0p1p2p3} that
\[
T(w,u,\ell )=(-1)^wq^{wu-u^2}G(w,u,\ell ;-\ell ,u-1,-\ell )=0, \quad \forall \ell \in\bbZ_{>0}.
\]
Together with $T(w, u, 0)=0$ (for $u>0$) from Corollary~\ref{cor:x=0}, this proves Theorem \ref{thm:T=0}.
\qed

%%%%%%%%
%%%%%%%%
\appendix

\section{More reductions}  %\eqref{eq:serre11odd}--\eqref{eq:serre11evenodd}}
  \label{sec:appendix}

In this appendix, we provide details on the proofs of the identities \eqref{eq:serre11odd}--\eqref{eq:serre11evenodd}, which are modeled on the proof of \eqref{eq:serre11F}.

\subsection{Proof of the identity \eqref{eq:serre11odd}}

Recall $a_{12}=-2m$. Thanks to Remark~\ref{rem:u=0}, in order to prove the identity \eqref{eq:serre11odd}, it suffices to prove
\begin{align}\label{eqn:QSP11oddF}
\sum_{n=0}^{2m+1} (-1)^n  B^{(n)}_{1,\odd}F_2 B^{(2m+1-n)}_{1,\odd} \onestar_{2\la-1}=0,
\qquad \forall \la \in \Z.
\end{align}

Similar to \eqref{eq:evev}, we can show that
\begin{align}  \label{eq:F22}
%\displaybreak
&\small \sum_{n=0}^{2m+1} (-1)^n  B^{(n)}_{1,\odd}F_2 B^{(2m+1-n)}_{1,\odd}\onestar_{2\la-1} \\ \notag
&\small = \Big\{\sum^{2m}_{n=0,2\mid n}\sum_{c=0}^{m-\frac{n}{2}}\sum_{e=0}^{\frac{n}{2}} \sum_{a=0}^{2m+1-n-2c}\sum_{d=0}^{n-2e}\sum^{\min\{a,n-2e-d\}}_{r=0}q_1^{(a+c+d+e)(2m+2-n-2\la-2a-2c-2d-2e)-c+d+e}  \\ \notag
&\small \quad \cdot \qbinom{a+d-r}{d}_{q_1}\qbinom{2m+3-n-2\la-2e-d-3a-4c}{r}_{q_1} \qbinom{m-\frac{n}{2}-\la-c-a+1}{c}_{q_1^2}
\\ \notag
&\small \quad \cdot \qbinom{m+1-\frac{n}{2}-\la-e-d-2a-2c}{e}_{q_1^2} \\ \notag
&\small \quad - \sum^{2m+1}_{n=1,2\nmid n}\sum_{c=0}^{m+\frac{1-n}{2}}\sum_{e=0}^{\frac{n-1}{2}} \sum_{a=0}^{2m+1-n-2c}\sum_{d=0}^{n-2e}\sum^{\min\{a,n-2e-d\}}_{r=0}q_1^{(a+c+d+e)(2m+2-n-2\la-2a-2c-2d-2e)+d} \\ \notag
&\small \quad \cdot \qbinom{a+d-r}{d}_{q_1}\qbinom{2m+3-n-2\la-2e-d-3a-4c}{r}_{q_1} \qbinom{m+\frac{1-n}{2}-\la-c-a}{c}_{q_1^2}
\\ \notag
&\small \quad \cdot \qbinom{m+\frac{3-n}{2}-\la-e-d-2a-2c}{e}_{q_1^2} \Big\} E_1^{(a+d-r)}F_1^{(n-2e-d-r)}F_2F_1^{(2m+1-n-2c-a)}\onestar_{2\la-1}. \notag
\end{align}

We introduce variables $\ell =a+d-r$, and $y =n-2e-d-r$. As $n+2c+a-\ell -y$ is even, set $n+2c+a-\ell -y =2u$ for $u\in\Z$. Then we have
$a=\ell +y+2u-2c-n$, $d=n+c-e-u-y$ and $u=e+c+r$. Observe the monomials on the right-hand side of the equation \eqref{eq:F22} above are of the form $E_1^{(\ell)}F_1^{(y)}F_2F^{(2m+1-\ell-y-2u)}_1\onestar_{2\la}$, for $\ell, y, u \in \N$.

If $u=0$ and $\ell =0$, then $e=c=r=a=d=0$, collecting the corresponding monomials in \eqref{eq:F22} gives us $\sum^{2m+1}_{n=0}(-1)^nF^{(n)}_1F_2F^{(2m+1-n)}_1\onestar_{2\la}=0$, by the $q$-Serre relation \eqref{eq:serre4}.

If $u$ and $\ell$ are not both $0$, the monomial $E_1^{(\ell )}F_1^{(y)}F_2F^{(2m+1-\ell -y-2u)}_1\onestar_{2\la-1}$ has coefficient given by $q_1^{(\ell +u)(2m+2-2\la-2\ell -3u-y)} S'(y,u,\ell, \la)$, where
\begin{align}\label{eq:T'(x,y,u)}
\displaybreak
S' & (y,u,\ell, \la) =\sum^{2m}_{n=0,2\mid n}\sum_{\substack{c,e,r\geq0 \\ c+e+r=u}}
q_1^{(u+y-n)(\ell +u-1)} \\
& \cdot \qbinom{\ell }{-u-y-e+c+n}_{q_1}
 \qbinom{2m+3-2\la-5u-3\ell -2y-e+c+n}{r}_{q_1} \notag \\
& \cdot \qbinom{m-\la-2u-\ell -y+c+1+\frac{n}{2}}{c}_{q_1^2}
 \qbinom{m+1-\la-2\ell -y-3u+\frac{n}{2}+c}{e}_{q_1^2}
 \notag \\
&- \sum^{2m+1}_{n=1,2\nmid n} \sum_{\substack{c,e,r\geq0 \\ c+e+r=u}}
q_1^{(u+y-n)(\ell +u-1)+c-e} \notag \\
&\cdot \qbinom{\ell }{-u-y-e+c+n}_{q_1}
\qbinom{2m+3-2\la-5u-3\ell -2y-e+c+n}{r}_{q_1} \notag \\
&\cdot  \qbinom{m-\la-2u-\ell -y+c+\frac{n+1}{2}}{c}_{q_1^2}
\qbinom{m+1-\la-2\ell -y-3u+\frac{n+1}{2}+c}{e}_{q_1^2} . \notag
\end{align}

Using new variables $t=-u-y-e+c+n$ and $w=2m+3-2\la-2\ell -4u-y$, we rewrite $S'(y,u,\ell, \la) =T'(w,u, \ell)$, where
\begin{align*}
T'(w,u,\ell )
&=  \sum_{\substack{c,e,r\geq0 \\ c+e+r=u}}
 \sum^{\ell }_{\substack{t=0\\ 2\nmid(t+w-r)}}
 q_1^{-t(\ell +u-1)-(\ell +u-1)(e-c)}\\ \notag
&\qquad\cdot \qbinom{\ell }{t}_{q_1}\qbinom{w+t-\ell }{r}_{q_1} \qbinom{\frac{w+t-r-1}{2}+u}{c}_{q_1^2}\qbinom{\frac{w+t-r-1}{2}-\ell }{e}_{q_1^2}\\  \notag
&\quad -\sum_{\substack{c,e,r\geq0 \\ c+e+r=u}}
\sum^{\ell }_{\substack{t=0\\ 2\mid(t+w-r)}}
q_1^{-t(\ell +u-1)-(\ell +u)(e-c)}\\ \notag
&\qquad \cdot \qbinom{\ell }{t}_{q_1}  \qbinom{w+t-\ell }{r}_{q_1}  \qbinom{-1+\frac{w+t-r}{2}+u}{c}_{q_1^2}
\qbinom{\frac{w+t-r}{2}-\ell }{e}_{q_1^2}.
\notag
\end{align*}
We observe that $T'(w,u,\ell ) =-T(w,u,\ell )|_{q\mapsto q_1}$  as defined in \eqref{eq:Twux}.
Therefore, \eqref{eqn:QSP11oddF} follows from Theorem \ref{thm:T=0}. The identity \eqref{eq:serre11odd} is proved.

\subsection{Proof of the identity \eqref{eq:serre11oddeven}}

Let $a_{12}=1-2m$. To prove the identity \eqref{eq:serre11oddeven}, it suffices to prove that
\begin{align}\label{eqn: QSP11oddevenF}
\sum_{n=0}^{2m} (-1)^n  B^{(n)}_{1,\odd}F_2 B^{(2m-n)}_{1,\ev} \onestar_{2\la}=0,
\qquad \forall \la \in \Z.
\end{align}

Similar to \eqref{eq:evev}, by computations we obtain
\begin{eqnarray}
\label{eq:F10}
 &&\small \sum_{n=0}^{2m} (-1)^n  B^{(n)}_{1,\odd}F_2 B^{(2m-n)}_{1,\ev} \onestar_{2\la} \\ \notag
& = & \Big\{ \small \sum^{2m}_{n=0,2\mid n}\sum_{c=0}^{m-\frac{n}{2}}\sum_{e=0}^{\frac{n}{2}} \sum_{a=0}^{2m-n-2c}\sum_{d=0}^{n-2e}\sum^{\min\{a,n-2e-d\}}_{r=0}q_1^{(a+c+d+e)(2m-n-2\la-2a-2c-2d-2e)-c+d+e}\\ \notag
&&\small \cdot \qbinom{a+d-r}{d}_{q_1}\qbinom{2m-2\la-3a-4c-d-2e-n+1}{r}_{q_1} \qbinom{m-\frac{n}{2}-c-a-\la}{c}_{q_1^2}
\\ \notag
&&\small  \cdot \qbinom{m-\la-e-d-2a-2c-\frac{n}{2}}{e}_{q_1^2}\\
&-&\small \sum^{2m}_{n=1,2\nmid n}\sum_{c=0}^{m-\frac{1+n}{2}}\sum_{e=0}^{\frac{n-1}{2}} \sum_{a=0}^{2m-n-2c}\sum_{d=0}^{n-2e}\sum^{\min\{a,n-2e-d\}}_{r=0}q_1^{(a+c+d+e)(2m-n-2\la-2a-2c-2d-2e)+d}\notag\\ \notag
&&\small \cdot \qbinom{a+d-r}{d}_{q_1}\qbinom{2m-2\la-3a-4c-d-2e-n+1}{r}_{q_1} \qbinom{m+\frac{1-n}{2}-c-a-\la-1}{c}_{q_1^2}
\\
&&\small \cdot \qbinom{m-\la-e-d-2a-2c-\frac{n+1}{2}+1}{e}_{q_1^2} \Big\}
E_1^{(a+d-r)}F_1^{(n-2e-d-r)}F_2F_1^{(2m-n-2c-a)}\onestar_{2\la}. \notag
\end{eqnarray}

Introduce new variables $\ell =a+d-r$, $y =n-2e-d-r$, and $2u =n+2c+a-\ell -y$. Then we have
$a=\ell +y+2u-2c-n$, $d=n+c-e-u-y$, and $r=u-e-c$. Observe the monomials on the right-hand side of the equation \eqref{eq:F10} above are of the form $E_1^{(\ell)}F_1^{(y)}F_2F^{(2m-\ell-y-2u)}_1\onestar_{2\la}$, for $\ell, y, u \in \N$.

%Let $Q=q^{(\ell +u)(2m-2\la-2\ell -3u-y)}$.
For $u, \ell \in \N$, not both $0$, the monomial $E_1^{(\ell )}F_1^{(y)}F_2F^{(2m-\ell -y-2u)}_1\onestar_{2\la}$ has coefficient given by $q_1^{(\ell +u)(2m-2\la-2\ell -3u-y)} S''(y,u,\ell, \la)$, where
\begin{align}\label{eq:T1(x,y,u)}
S'' & (y,u,\ell, \la)
=\sum^{2m}_{n=0,2\mid n}\sum_{\substack{c,e,r\geq0 \\ c+e+r=u}}
q_1^{-(n+u+y)(\ell +u-1)}  \\
&\quad \cdot \qbinom{\ell }{-u-y-e+c+n}_{q_1}
 \qbinom{2m+1-2\la-5u-3\ell -2y-e+c+n}{u-e-c}_{q_1} \notag \\
& \quad \cdot \qbinom{m-\la-\ell -y-2u+c+\frac{n}{2}}{c}_{q_1^2}
 \qbinom{m-\la-2\ell -y-3u+\frac{n}{2}+c}{e}_{q_1^2}  \notag
\\
& \quad -\sum^{2m}_{n=1,2\nmid n}\sum_{\substack{c,e,r\geq0 \\ c+e+r=u}}
q_1^{-(n+u+y)(\ell +u-1)+c-e} \notag \\
& \quad \cdot \qbinom{\ell }{-u-y-e+c+n}_{q_1}
\qbinom{2m+1-2\la-5u-3\ell -2y-e+c+n}{u-e-c}_{q_1} \notag \\
& \quad\cdot \qbinom{m-\la-\ell -y-2u+c+\frac{n+1}{2}-1}{c}_{q_1^2}
 \qbinom{m-\la-2\ell -y-3u+\frac{n+1}{2}+c}{e}_{q_1^2}. \notag
\end{align}
Using new variables $t=-u-y-e+c+n$ and $w=2m+1-2\la-2\ell -4u-y$, we can show that $S''(y,u,\ell, \la) =-T(w,u,\ell ) |_{q\mapsto q_1}$  as defined in \eqref{eq:Twux}, and \eqref{eqn: QSP11oddevenF} follows. The identity \eqref{eq:serre11oddeven} is proved.

\subsection{Proof of the identity \eqref{eq:serre11evenodd}}

Let $a_{12}=1-2m$. To prove the identity \eqref{eq:serre11evenodd}, it suffices to prove that
\begin{align}\label{eqn: QSP11evenoddF}
\sum_{n=0}^{2m} (-1)^n  B^{(n)}_{1,\ev}F_2 B^{(2m-n)}_{1,\odd} \onestar_{2\la-1} =0,
\qquad \forall \la \in \Z.
\end{align}

Similar to \eqref{eq:evev}, by computations we can show
\begin{align}
  \label{eq:F11}
%\displaybreak
&\small \sum_{n=0}^{2m} (-1)^n  B^{(n)}_{1,\ev}F_2 B^{(2m-n)}_{1,\odd} \onestar_{2\la-1} \\
&= \notag
\small \Big\{\sum^{2m}_{n=0,2|n}\sum_{c=0}^{m-\frac{n}{2}}\sum_{e=0}^{\frac{n}{2}} \sum_{a=0}^{2m-n-2c}\sum_{d=0}^{n-2e}\sum^{\min\{a,n-2e-d\}}_{r=0}q_1^{(a+c+d+e)(2m+1-n-2\la-2a-2c-2d-2e)+d}\\  \notag
&\small  \cdot \qbinom{a+d-r}{d}_{q_1}\qbinom{2m+2-2\la-3a-4c-d-2e-n}{r}_{q_1} \qbinom{m-\frac{n}{2}-c-a-\la}{c}_{q_1^2}
\\  \notag
&\small \cdot \qbinom{m+1-\la-e-d-2a-2c-\frac{n}{2}}{e}_{q_1^2}
\\
&-\small \sum^{2m}_{n=0,2\nmid n}\sum_{c=0}^{m-\frac{n+1}{2}}\sum_{e=0}^{\frac{n-1}{2}} \sum_{a=0}^{2m-n-2c}\sum_{d=0}^{n-2e}\sum^{\min\{a,n-2e-d\}}_{r=0}q_1^{(a+c+d+e)(2m+1-n-2\la-2a-2c-2d-2e)+d+e-c}\notag\\ \notag
&\small \cdot \qbinom{a+d-r}{d}_{q_1}\qbinom{2m+2-2\la-3a-4c-d-2e-n}{r}_{q_1} \qbinom{m-\frac{n-1}{2}-c-a-\la}{c}_{q_1^2}
\\ \notag
&\small \cdot \qbinom{m-\la-e-d-2a-2c-\frac{n-1}{2}}{e}_{q_1^2} \Big \}E_1^{(a+d-r)}F_1^{(n-2e-d-r)}F_2F_1^{(2m-n-2c-a)}\onestar_{2\la-1}. \notag
\end{align}

We change variables $\ell =a+d-j$, $y =i-2e-d-j$, $2m-i-2c-a=2m-\ell -y-2u$. Then we have
$a=\ell +y+2u-2c-i$, $d=i+c-e-u-y$, $j=u-e-c$. Observe the monomials on the right-hand side of the equation \eqref{eq:F11} above are of the form $E_1^{(\ell)}F_1^{(y)}F_2F^{(2m-\ell-y-2u)}_1\onestar_{2\la}$, for $\ell, y, u \in \N$.
For $u, \ell \in \N$, not both $0$, the monomial $E_1^{(\ell )}F_1^{(y)}F_2F^{(2m-\ell -y-2u)}_1\onestar_{2\la-1}$ is $q_1^{(\ell +u)(2m+1-2\la-2\ell -3u-y)} S'''(y,u,\ell, \la)$, where
\begin{align}
 \label{eq:T'1(x,y,u)}
\small S''' & (y,u,\ell, \la)
\small =\sum^{2m}_{n=0,2\mid n}\sum_{\substack{c,e,r\geq0 \\ c+e+r=u}}
 q_1^{-(n+u+y)(\ell +u-1)+c-e} \\
 &\small \quad\cdot \qbinom{\ell }{-u-y-e+c+n}_{q_1}
 \qbinom{2m+2-2\la-5u-3\ell -2y-e+c+n}{r}_{q_1} \notag  \\
 &\small \quad \cdot \qbinom{m-\la-\ell -y-2u+c+\frac{n}{2}}{c}_{q_1^2}
 \qbinom{m+1-\la-2\ell -y-3u+\frac{n}{2}+c}{e}_{q_1^2} \notag \\
&\small \quad - \sum^{2m}_{n=1,2\nmid n}\sum_{\substack{c,e,r\geq0 \\ c+e+r=u}}
 q_1^{-(n+u+y)(\ell +u-1)}\qbinom{\ell }{-u-y-e+c+n}_{q_1} \notag
 \\
& \small \quad \cdot \qbinom{2m+2-2\la-5u-3\ell -2y-e+c+n}{r}_{q_1}
 \qbinom{m-\la-\ell -y-2u+c+\frac{n+1}{2}}{c}_{q_1^2}  \notag
\\
& \small \quad \cdot \qbinom{m-\la-2\ell -y-3u+\frac{n+1}{2}+c}{e}_{q_1^2}. \notag
\end{align}
It is easy to note that $S'''(y,u,\ell, \la)=q_1^{-(2u+2y)(\ell +u-1)}S(y,u,\ell, \la)$ (see \eqref{eq:T(x,y,u;w)}). Hence \eqref{eqn: QSP11evenoddF} follows from \eqref{eq:S=T} and Theorem \ref{thm:T=0}. The identity \eqref{eq:serre11evenodd} is proved.

Summarizing, in this appendix we have completed the proofs of the identities \eqref{eq:serre11odd}--\eqref{eq:serre11evenodd}.

%%%%%%%%
%%%%%%%%

\end{document}